\documentclass{birkjour}
\usepackage{fullpage,amsthm,amsmath}
\usepackage{amsfonts,amsmath,amsthm,amscd,amssymb,mathrsfs}
\usepackage{mathtools}
\usepackage{enumerate}
\usepackage{verbatim}
\usepackage{hyperref}
\usepackage{import}
\usepackage{enumitem}
\usepackage{caption}
\usepackage[T1]{fontenc}
\usepackage{bigfoot} % to allow verbatim in footnote
\usepackage[numbered,framed]{matlab-prettifier}
\usepackage{filecontents}

\newcommand{\C}{\mathbb{C}}
\newcommand{\R}{\mathbb{R}}
\newcommand{\Z}{\mathbb{Z}}

\newcommand{\N}{\mathbb{N}}

\newcommand{\LL}{\mathcal{L}}

\newcommand{\E}{\mathbb{E}}

\DeclareMathOperator{\rank}{rank}

\DeclareMathOperator*{\esup}{ess\,sup}

\DeclareMathOperator*{\intr}{int}

\DeclareMathOperator*{\supp}{supp}

\DeclareMathOperator*{\dist}{dist}

\DeclareMathOperator*{\diam}{diam}

\newtheoremstyle{break}
  {}{}{}{}
  {\bfseries}{}% % Note that final punctuation is omitted.
  {\newline}{}

\theoremstyle{plain}
\newtheorem{theorem}{Theorem}[section]
\newtheorem{definition}[theorem]{Definition}
\newtheorem{proposition}[theorem]{Proposition}
\newtheorem{lemma}[theorem]{Lemma}
\newtheorem{corollary}[theorem]{Corollary}
\newtheorem{remark}[theorem]{Remark}
\newtheorem*{theorem*}{Theorem}

\theoremstyle{remark}
\newtheorem{example}[theorem]{Example}

\theoremstyle{break}

\newcommand{\abs}[1]{\left\lvert#1\right\rvert}
\newcommand{\norm}[1]{\left\lVert#1\right\rVert}
\newcommand{\tnorm}[1]{{\left\vert\kern-0.25ex\left\vert\kern-0.25ex\left\vert #1
    \right\vert\kern-0.25ex\right\vert\kern-0.25ex\right\vert}}

\DeclarePairedDelimiterX{\inp}[2]{\langle}{\rangle}{#1, #2}

\newcommand*\conj[1]{\overline{#1}}

\newcommand*{\cl}[1]{\overline{#1}}

\newcommand\restr[2]{{% we make the whole thing an ordinary symbol
  \left.\kern-\nulldelimiterspace % automatically resize the bar with \right
  #1 % the function
  \vphantom{\big|} % pretend it's a little taller at normal size
  \right|_{#2} % this is the delimiter
  }}

\makeatletter
\newcommand*{\intf}{%
  \@ifnextchar_{\intf@sub}{%
    \@ifnextchar^{\intf@sup}{%
      \intf@{}{}%
    }%
  }%
}
\def\intf@sub_#1{%
  \@ifnextchar^{%
    \intf@sub@sup{#1}%
  }{%
    \intf@{#1}{}%
  }%
}
\def\intf@sup^#1{%
  \@ifnextchar_{%
    \intf@sup@sub{#1}%
  }{%
    \intf@{}{#1}%
  }%
}
\def\intf@sub@sup#1^#2{\intf@{#1}{#2}}
\def\intf@sup@sub#1_#2{\intf@{#2}{#1}}
\def\intf@#1#2#3d#4{%
  \int
  \ifx\\#1\\\else _{#1}\fi
  \ifx\\#2\\\else ^{#2}\fi
  \!#3\,\mathrm{d}#4%
}
\makeatother

\lstMakeShortInline"

\newcommand{\wnorm}[1]{\abs{#1}}
\newcommand{\esigma}{\sigma_{ess}}
\newcommand{\erho}{\rho_{ess}}

\newcommand{\lnorm}[1]{\abs{#1}_{L^1}}

\newcommand{\clo}[1]{\overline{#1}}

\DeclareMathOperator{\Var}{Var}
\DeclareMathOperator{\BV}{BV}

\newcommand{\Va}{\mathbb{V}_{\beta}}
\newcommand{\funcosc}{\textup{\textrm{osc}}}
\newcommand{\osc}[1]{\textup{\textrm{osc}}\left(#1\right)}
\newcommand{\avg}[2]{\hat{#1}_{#2}}

\newcommand{\bfone}{\mathbf{1}}
\DeclareMathOperator{\Id}{Id}

\newcommand\blfootnote[1]{%
  \begingroup
  \renewcommand\thefootnote{}\footnote{#1}%
  \addtocounter{footnote}{-1}%
  \endgroup
}

\newtheorem{innercustomthm}{Theorem}
\newenvironment{customthm}[1]
 {\renewcommand\theinnercustomthm{#1}\innercustomthm}
 {\endinnercustomthm}

\title{Stability and approximation of statistical limit laws for multidimensional piecewise expanding maps}

\author{Harry Crimmins}
\email{h.crimmins@unsw.edu.au}
\address{School of Mathematics and Statistics \br
University of New South Wales \br
Sydney NSW 2052, Australia}

\author{Gary Froyland}
\email{g.froyland@unsw.edu.au}
\address{School of Mathematics and Statistics \br
University of New South Wales \br
Sydney NSW 2052, Australia}

\date{January 8, 2019}

\begin{document}

\begin{abstract}
The unpredictability of chaotic nonlinear dynamics leads naturally to statistical descriptions, including probabilistic limit laws such as the central limit theorem and large deviation principle.
A key tool in the Nagaev-Guivarc'h spectral method for establishing statistical limit theorems is a ``twisted'' transfer operator.
In the abstract setting of Keller-Liverani \cite{keller1999stability} we prove that derivatives of all orders of the leading eigenvalues and eigenprojections of the twisted transfer operators with respect to the twist parameter are stable when subjected to a broad class of perturbations.
As a result, we demonstrate stability of the variance in the central limit theorem and the rate function from a large deviation principle with respect to deterministic and stochastic perturbations of the dynamics and perturbations induced by numerical schemes.
We apply these results to piecewise expanding maps in one and multiple dimensions, including new convergence results for Ulam projections on quasi-H\"older spaces.
\end{abstract}

\maketitle

\blfootnote{The authors' emails are h.crimmins@unsw.edu.au and g.froyland@unsw.edu.au, respectively.}

\newpage

%\tableofcontents

%\newpage

\section{Introduction}

Transfer operators have proved to be powerful tools for the analysis of dynamical systems possessing some expanding properties.
If $(X,m)$ is a probability space and $T : X \to X$ is a non-singular transformation, one aims to select a Banach space $(B,\norm{\cdot})$, $B\subset L^1(m)$ that is compatible with the dynamics of $T$ in the sense that the transfer operator $\mathcal{L}:B\to B$ is quasi-compact.
Quasi-compactness implies a variety of desirable phenomena, including a finite number of absolutely continuous invariant probability measures (ACIMs) with densities lying in $B$, and under a mixing condition, a unique ACIM with exponential decay of correlations \cite{hofbauer1982ergodic, liverani1995decay}.

A natural question is how the spectral data corresponding to the isolated eigenvalues of $\mathcal{L}$ behave when either the map $T$ or the operator $\mathcal{L}$ is perturbed.
It is well known that many interesting perturbations are not close to $\mathcal{L}$ in the operator norm induced by the norm of $B$, and therefore standard operator perturbation theory, e.g. \cite{kato1966perturbation} cannot be employed.
For piecewise monotonic maps of the interval, with $\mathcal{L}$ acting on the space of functions of bounded variation $\BV$, Keller \cite{keller1982stochastic} introduced a ``triple norm'' $\tnorm{A}=\sup_{{\rm var}(f)+\lnorm{f}\le 1}\lnorm{Af}$ for $A:\BV\to \BV$.
For a variety of deterministic, stochastic, and numerical perturbations, the resulting perturbed transfer operator  $\mathcal{L}_\epsilon$ is close to $\mathcal{L}$ in $\tnorm{\cdot}$.
These results were abstracted in a seminal paper by Keller and Liverani \cite{keller1999stability}, where $\lnorm{\cdot}$ becomes a ``weak'' norm $\abs{\cdot}$ and ${\rm var}(\cdot)+\lnorm{\cdot}$ becomes a ``strong'' norm $\|\cdot\|$ on a Banach space $B$, with the closed unit $\norm{\cdot}$-ball compact in $\abs{\cdot}$.
Using this abstract setup, with additional conditions on the growth of the norms of iterates of $\mathcal{L}_\epsilon$ in the weak and strong norms, \cite{keller1999stability} proved stability of the isolated spectral data.

The use of spectral theory has also been a remarkably successful strategy for establishing statistical laws for dynamical systems such as central limit theorems \cite{RE83,Broise,HennionHerve,AyyerLiveraniStenlund} and  large deviation principles \cite{HennionHerve,rey-young08}, as well as local central limit theorems \cite{RE83,HennionHerve,gouezel05}, Berry-Esseen theorems \cite{GH88,gouezel05}, and vector-valued almost-sure invariance principles \cite{melbourne_nicol09,gouezel10}.
We refer the reader to the excellent survey \cite{gouezel2015limit} and the references therein.
Assuming that $T$ possesses a unique ACIM $\mu$ with $\frac{\mathrm{d}\mu}{\mathrm{d}m}\in B$, and
given a real-valued observable $g\in B\cap L^\infty(m)$, we can define a stochastic process $Z_k:=\{g\circ T^k\}_{k\ge 0}$, stationary with respect to $\mu$.
The process $Z_k$ has mean $\bar{g}=\intf_X g d\mu$, satisfies a central limit theorem with time-asymptotic variance $\sigma_g^2$ about the mean $\bar{g}$, and has exponentially decaying probabilities for large deviations from $\bar{g}$, quantified by a rate function $r_g(s)=-\lim_{n\to\infty}\frac{1}{n}\log \mu(\frac{1}{n}\sum_{i=0}^{n-1}g\circ T^{n}>\bar{g}+s)$.
The quantities $\sigma_g$ and $r_g$ are accessible via derivatives of the leading eigenvalue of an analytically ``twisted'' transfer operator $\mathcal{L}(z)f:=\mathcal{L}(e^{z g}f)$, taken with respect to the twist parameter $z$.

In this work we bring together these two threads to prove stability of statistical laws under a broad class of perturbations satisfying abstract conditions we denote by (KL); see Definition \ref{KLdefn} for a formal definition.
We extend the general spectral stability approach of \cite{keller1999stability} to twisted transfer operators, providing access to stability results for statistical laws obtained via spectral theory across a range of perturbations in a general abstract setting.
For the specific problem of numerical approximation of statistical quantities, our flexible setup allows a variety of projection methods, enabling the projection to be tailored to the particular class of dynamics to achieve the most efficient numerical scheme.

We outline below our first main abstract result (see Theorem \ref{theorem:convergence_of_eigendata_derivs} for a more precise formal statement of the result) concerning stability of the derivatives of spectral data of twisted quasi-compact operators, taken with respect to the twist parameter.

\begin{customthm}{A}\label{thmA} Let $A_0$ be a quasi-compact operator with simple leading eigenvalue $\lambda_0$ and quasi-compact decomposition $\lambda_0\Pi_0+N_0$, and let $\{A_\epsilon\}_{\epsilon\ge 0}$ be a family of perturbations of type (KL).
Let  $A_\epsilon(z)=A_\epsilon M(z)$ for a compactly $\wnorm{\cdot}$-bounded twist $M(z)$ (see Definition \ref{def:twist}).
For $z$ in a sufficiently small neighbourhood of the origin and $\epsilon$ sufficiently small one has:
\begin{enumerate}
\item a quasi-compact decomposition $A_\epsilon(z)=\lambda_\epsilon(z)\Pi_\epsilon(z)+N_\epsilon(z)$, where $\lambda_\epsilon(z), \Pi_\epsilon(z)$ and $N_\epsilon(z)$ depend analytically on $z$,
\item for each $n \in \N$ the following convergence for $n^{th}$-order derivatives with respect to $z$ as $\epsilon \to 0$:
  \begin{enumerate}
    \item $\lambda_\epsilon^{(n)}(\cdot)$ converges compactly to $\lambda_0^{(n)}(\cdot)$,
    \item $\Pi_\epsilon^{(n)}(\cdot)$ converges compactly to $\Pi_0^{(n)}(\cdot)$ in $\tnorm{\cdot}$,
    \item $N_\epsilon^{(n)}(\cdot)$ converges compactly to $N_0^{(n)}(\cdot)$ in $\tnorm{\cdot}$.
  \end{enumerate}
  \end{enumerate}
\end{customthm}

In fact, in Theorem \ref{theorem:convergence_of_eigendata_derivs} we prove a stronger H{\"o}lder estimate for the convergence in (a), (b) and (c).

When $A_0$ is the transfer operator $\mathcal{L}$ corresponding to a uniformly expanding, piecewise $\mathcal{C}^2$ map $T:X\to X$, where $X$ may be an interval, or a more complicated higher dimensional domain as in \cite{saussol2000absolutely}, three specific examples of the types of perturbations (KL) we consider include:
  \begin{enumerate}
    \item[NP:] \textsl{Numerical approximations of $\mathcal{L}$ by some finite-rank operator, such as in Ulam's method \cite{li}.  Given a partition of $X$ into $n$ connected elements of diameter less than $\epsilon$, define $\mathbb{E}_n$ to be the conditional expectation operator with respect to this partition and let $\mathcal{L}_\epsilon=\mathbb{E}_n\circ\mathcal{L}$.}
  \end{enumerate}
For untwisted operators, stability of the ACIM with respect to (NP) has been proven in one dimension \cite{li} with $B$ the space of bounded variation $\BV$ and in multiple dimensions \cite{DZ96,murraymultidim} with $B$ the space of generalised bounded variation. For the one dimensional case we mention also the works \cite{gora1984, keller1982stochastic}, which generalise the results of \cite{li}.
In Section \ref{sec:1dulampert} we briefly recount the fact that the family of operators $\mathcal{L}_\epsilon$ satisfies (KL) for interval maps.
In Section \ref{sec:piecewise_expanding_multi_dim} we prove the stability of the ACIM when $B$ is the Banach algebra of quasi-H\"older functions and show that the family of operators $\mathcal{L}_\epsilon$ satisfies (KL) for multidimensional piecewise expanding maps.
  \begin{enumerate}
    \item[SP:] \textsl{Stochastic perturbations that arise from the convolution of the Perron-Frobenius operator with an appropriate bistochastic, nonnegative kernel} $K_\epsilon(x,y)$:  $\mathcal{L}_\epsilon f(x)=\intf (\LL f)(y)K_\epsilon(y,x) dm(y)$.
  \end{enumerate}
With $X=[0,1]$ and $B$ the space of functions of bounded variation, if for $U\subset X$ the measure $m_\epsilon(U):=\intf_{U \times U} K_\epsilon dm\times m$ converges weakly to $m$ lifted to the diagonal of $X\times X$ and under mild monotonicity conditions, Corollary 17 \cite{keller1982stochastic}, proves stability of the ACIM with respect to (SP) for piecewise expanding $T$ with $\abs{T'}>2$. Under these conditions, the family of operators $\mathcal{L}_\epsilon$ also satisfies (KL).
In Section \ref{sec:piecewise_expanding_multi_dim} we provide a new proof of stability of the ACIM when $B$ is the space of quasi-H\"older functions and show that the family of operators $\mathcal{L}_\epsilon$ satisfies (KL).
  \begin{enumerate}
    \item[DP:] \textsl{Deterministic perturbations of $T$ in an appropriate metric space. For example, in a ``Skorohod''-type metric in the case of piecewise expanding maps on the interval \cite[Section 3]{keller1982stochastic}:
  \begin{equation}\begin{split}\label{staticmetric}
    d(T,T_\epsilon):=\inf\{\delta>0\,:\, &\exists U\subset [0,1],\, m(U)>1-\delta; \\
    &\exists \mbox{ a diffeomorphism } h:[0,1]\to[0,1], \\
    & T|_U=T_\epsilon\circ h|_U,\mbox{ and }\forall x\in U,\\ &|h(x)-x|<\delta,\, |(1/h'(x))-1|<\delta\}.
  \end{split}\end{equation}}
  \end{enumerate}
If $\lim_{\epsilon \to 0} d(T,T_\epsilon) = 0$ then the family of operators $\{\mathcal{L}_\epsilon\}_{\epsilon \ge 0}$ satisfies (KL) \cite[\S3]{keller1982stochastic}.

Given perturbations $\{\LL_\epsilon\}_{\epsilon \ge 0}$ satisfying (KL), we use the spectral formulae for the variance $\sigma_g^2$ and rate function $r_g$ to define perturbed versions of these quantities.
Let $\sigma_{g,\epsilon}^2=\lambda_\epsilon^{(2)}(0)$ be the perturbed variance and let the convex conjugate of the map $z \mapsto \log\lambda_\epsilon(z)$ be the perturbed rate function.
Below we outline our second main result (see Theorems \ref{thm:stability_of_variance}, \ref{thm:convergence_of_rate_functions} and Proposition \ref{prop:estimation_of_truncated_rate_func} for formal statements), which guarantees stability of the variance $\sigma_g^2$ and the rate function $r_g$ under the perturbations of type (KL).
We emphasise that we only need to check the perturbation conditions (KL) for the \emph{untwisted} ($z=0$) operators, and, as with Theorem \ref{thmA}, that we actually prove stronger H{\"o}lder estimates than the modes of convergence given below.

\begin{customthm}{B}\label{thmB}
Under the hypotheses of Theorem \ref{thmA} and the Nagaev-Guivarc'h spectral method, with $A_\epsilon=\mathcal{L}_\epsilon$, $M(z)f=e^{zg}f$, and $B$ a Banach algebra, one has
\begin{enumerate}
\item Stability of the variance:  $\lim_{\epsilon\to 0}\lambda_\epsilon^{(2)}(0)=\sigma_g^2$.
\item Stability of the rate function: If each $\LL_\epsilon$ is positive then for every compact subset $W$ of the domain of $r_g$ there exists a closed interval $V$ such that $\lim_{\epsilon \to 0} \sup_{z \in V} (sz - \ln(\lambda_\epsilon(z))) = r_g(s)$ uniformly for $s \in W$.
\end{enumerate}
\end{customthm}

Theorem \ref{thmB} guarantees convergence of numerical estimates of statistical laws produced by a variety of numerical schemes.
In this work we explore Ulam's method in detail, including new results on the quasi-H\"older spaces introduced by Saussol \cite{saussol2000absolutely}, which provide a flexible setting for piecewise expanding maps on complicated (possibly fractal) domains $X\subset \mathbb{R}^d$.
We prove that perturbations of the type (NP) and (SP) above are of type (KL) in the quasi-H\"older spaces and thus Theorems \ref{thmA} and \ref{thmB} apply to such perturbations.
While we have focussed on Ulam's method, our stability results can also be applied to other numerical methods of projection type (e.g.\ Galerkin methods projecting onto other locally or globally supported bases).

The stability of the variance under differentiable perturbations in the underlying dynamics has been considered for Anosov diffeomorphisms \cite{gouezel2006banach} and Lorenz flows \cite{bahsoun2018variance}. For perturbations of type (DP) the same has been proven for Lasota-Yorke maps \cite{keller2008continuity}.
There has also been recent interest in the numerical approximation of the variance $\sigma_g^2$.
Convergence of Ulam-based estimates for Lasota-Yorke maps has been rigorously implemented in one dimension \cite{bahsoun2016rigorous}, with a rigorous computation for the full-branched Lanford map, and extensions of the method for the intermittent Liverani-Saussol-Vaienti (LSV) map.
A periodic point algorithm for real analytic expanding interval maps and real analytic observables \cite{pollicott2017rigorous} exploits the analyticity to achieve rigorous estimates of the variance with rapid convergence.
For full-branched maps of the interval, Ulam's method has been replaced with a Galerkin method that utilises a Fourier or Chebyshev basis \cite{wormell2017spectral} to capitalise on the smoothness of the dynamics and increase the speed of convergence.
We note that it is possible to obtain the stability of the variance without using Theorem \ref{thmA}, but our proof of the stability of the rate function critically uses the compact convergence in Theorem \ref{thmA}, and to the best of our knowledge Theorem \ref{thmB} is the first rigorous stability and approximation result for the rate function for deterministic dynamical systems.
This point is discussed further in Remark \ref{remark:stability_of_variance}.
%following Proposition \ref{prop:2nd_deriv_twisted_eigenvalue_formula}).
%In this work we create a general abstract framework for obtaining stability results for the statistical laws of dynamical systems.
%Our results only depend on quasi-compactness and related properties of the Perron-Frobenius operator, and therefore cover a very wide range of perturbations and numerical schemes.

An outline of the paper is as follows.
Section \ref{sec:qc} reviews the stability theory of \cite{keller1999stability} for (untwisted) quasi-compact operators.
Section \ref{sec:twistedqc} introduces abstract compactly-bounded twists and proves the abstract stability of derivatives of the twisted spectral data (Theorem \ref{theorem:convergence_of_eigendata_derivs}).
Section \ref{sec:application_to_NG} reviews the Nagaev-Guivarc'h spectral method and verifies the applicability of the results of Section \ref{sec:twistedqc} in the setting of the method.
Using the stability of the derivatives of twisted spectral data, we prove the stability of the variance in Section \ref{subsec:stability_estimation_of_variance} (Theorem \ref{thm:stability_of_variance}), and stability of the rate function in Section \ref{subsec:stability_estimation_of_rate_function} (Theorem \ref{thm:convergence_of_rate_functions} and Proposition \ref{prop:estimation_of_truncated_rate_func}).
In Section \ref{subsec:stability_estimation_of_variance} we also obtain a derivative-free expression for the approximate variance $\lambda_\epsilon^{(2)}(0)$ for use in computations (Proposition \ref{prop:2nd_deriv_twisted_eigenvalue_formula}).
In Section \ref{sec:piecewise_expanding_one_dim} we provide detailed numerical experiments for a piecewise expanding interval map and four different observables, illustrating the approximation of the variance and rate function, and interpreting the results in terms of the interplay of dynamics and observable.
In Section \ref{sec:piecewise_expanding_multi_dim} we treat the multidimensional piecewise expanding maps of Saussol \cite{saussol2000absolutely} and prove that perturbations of the type (NP) and (SP) are compatible with the Banach algebra of quasi-H\"older functions, yielding stability and approximation results of statistical properties in a multidimensional setting.

\section{Spectral stability}\label{sec:abstract_twist_stability}

In Section \ref{sec:qc} we discuss known results concerning the stability of the spectrum of quasi-compact operators and in Section \ref{sec:twistedqc} we state our main results on spectral stability of ``twisted'' quasi-compact operators.
Let $(B, \norm{\cdot})$ be a Banach space over $\mathbb{C}$ and denote by $L(B)$ the Banach space of bounded linear operators from $B$ to itself.

\begin{definition}[Quasi-compactness]\label{def:quasicompactness}
  We say that $A \in L(B)$ is quasi-compact if there exists $A_1,A_2 \in L(B)$ such that $A = A_1 + A_2$, $\rank(A_1) < \infty$, $A_1 A_2 = A_2 A_1 = 0$ and $\rho(A_2) < \rho(A_1) = \rho(A)$. We say that $A$ is quasi-compact of diagonal type if $A_1$ is diagonalisable, and that $A$ is a simple quasi-compact operator if it is of diagonal type and $\rank(A_1) = 1$.
\end{definition}
If $A$ is quasi-compact of diagonal type then there exists a decomposition of $A$ of the form
\begin{equation*}
  A = \sum_{i=1}^s \lambda_i \Pi_i + N,
\end{equation*}
where each $\lambda_i$ is distinct and satisfies $\abs{\lambda_i} = \rho(A)$, each $\Pi_i$ is a finite-rank projection such that $\Pi_i \Pi_j = 0$ whenever $i \ne j$, and $N$ is a bounded operator such that $\rho(N) < \rho(A)$ and $N\Pi_i = \Pi_i N = 0$ for each $i$. We refer to such a decomposition as the \emph{quasi-compact decomposition of $A$}.

\subsection{Spectral stability of quasi-compact operators}
\label{sec:qc}
We review the spectral stability theory for quasi-compact operators from \cite{keller1999stability}.
Let $\wnorm{\cdot}$ be another norm on $B$ such that $\wnorm{\cdot} \le \norm{\cdot}$ and the closed, unit ball in $(B, \norm{\cdot})$ is compact in the topology induced by $\wnorm{\cdot}$.
We denote the norm on $L((B,\norm{\cdot}), (B,\wnorm{\cdot}))$ by $\tnorm{\cdot}$; that is
\begin{equation*}
  \tnorm{A} = \sup_{\norm{f} = 1} \wnorm{Af}.
\end{equation*}
The relevant families of operators are those satisfying the following condition.
\begin{definition}\label{KLdefn}
  We say that a family of operators $\{ A_\epsilon \}_{\epsilon \ge 0} \subseteq L(B)$ satisfies the Keller-Liverani (KL) condition if it verifies each of the following conditions.
  \begin{enumerate}[label=(KL\arabic*)]
    \item \label{en:kl_conv} There exists a monotone upper-semicontinuous function $\tau: [0, \infty) \to [0, \infty)$ such that $\tnorm{A_\epsilon - A_0} \le \tau(\epsilon)$ and $\tau(\epsilon) > 0$ whenever $\epsilon > 0$, and $\tau(\epsilon) \to 0$ as $\epsilon \to 0$.
    \item \label{en:kl_l1_bound} There exists $C_1, K_1 > 0$ such that $\wnorm{A_\epsilon^n} \le C_1 K_1^n$ for every $\epsilon \ge 0$ and $n \in \N$.
    \item \label{en:kl_ly_bound} There exists $C_2,C_3, K_2 > 0$ and $\alpha \in (0,1)$ such that
    \begin{equation*}
      \norm{A_\epsilon^n f} \le C_2 \alpha^n \norm{f} + C_3 K_2^n \wnorm{f}
    \end{equation*}
    for every $\epsilon \ge 0$, $f \in B$ and $n \in \N$.
  \end{enumerate}
\end{definition}

\begin{remark}
Denote the essential spectrum of $A$ by
\begin{equation*}
\esigma(A) = \left\{ \lambda \in \sigma(A) :
\begin{tabular}{c}
  $\lambda$ is not an eigenvalue of  $A$ with\\
  finite algebraic multiplicity
\end{tabular}\right\}.
\end{equation*}
If an individual operator $A$ satisfies \ref{en:kl_l1_bound} and \ref{en:kl_ly_bound} then its essential spectral radius $\erho(A)$ is bounded by $\alpha$ \cite{hennion1993theoreme,ionescu1950}. If, in addition, the spectral radius of $A$ is strictly greater than $\alpha$ then $A$ is quasi-compact.
\end{remark}

We will now summarise the conclusions of \cite{keller1999stability}.

\begin{proposition}[\cite{keller1999stability}]\label{prop:keller_liverani_summary}
  Let $\{ A_\epsilon \}_{\epsilon \ge 0} \subseteq L(B)$ satisfy (KL), where $A_0$ is a simple quasi-compact operator with decomposition $A_0 = \lambda_0 \Pi_0 + N_0$ and $\alpha < \abs{\lambda_0}$. For sufficiently small $\delta > 0$ and each $r$ such that $\max\{\alpha,\rho(N_0)\} < r < \abs{\lambda_0}$ there exists $\epsilon_{\delta,r} > 0$ such that $A_\epsilon$ is a simple quasi-compact operator with decomposition $A_\epsilon = \lambda_\epsilon \Pi_\epsilon + N_\epsilon$ whenever $\epsilon \in [0, \epsilon_{\delta,r}]$. Furthermore, for each $\epsilon \in [0, \epsilon_{\delta,r}]$ the spectral data of $A_\epsilon$ satisfies
  \begin{equation*}
    \lambda_\epsilon \in B_\delta(\lambda_0), \text{ and } \rho(N_\epsilon) < r,
  \end{equation*}
  in addition to the following H{\"o}lder estimate: there exists $C$ such that for all $\epsilon$ sufficiently small one has
  \begin{equation*}%\label{eq:keller_liverani_summary}
    \max\left\{\abs{\lambda_\epsilon - \lambda_0}, \tnorm{\Pi_\epsilon - \Pi_0}, \tnorm{N_\epsilon - N_0}\right \} \le C\tau(\epsilon)^\eta,
  \end{equation*}
  where $\eta := \frac{\log(r / \alpha)}{\log(\max\{K_1, K_2\}/ \alpha)}$.
\end{proposition}

%\begin{remark}
%  While it is improper to refer to \eqref{eq:keller_liverani_summary} as a H{\"o}lder estimate, in view of \ref{en:kl_conv} it is a sensible convention and so we will continue to refer to such bounds as H{\"o}lder estimates out of convenience.
%\end{remark}

\subsection{Spectral stability of twisted quasi-compact operators}
\label{sec:twistedqc}

We begin by defining an abstract twist.
\begin{definition}[Twist]\label{def:twist}
  If $M: D \to L(B)$ is analytic on an open neighbourhood $D \subseteq \C$ of 0 and $M(0)$ is the identity, then we call $M$ a \emph{twist}. If $A \in L(B)$ then the operators $A(z) := A M(z)$ are said to be \emph{twisted} by $M$. We say that $M$ is compactly $\wnorm{\cdot}$-bounded if for every compact $V \subseteq D$ there exists $C_{M,V} > 0$ such that
  \begin{equation*}
    \sup_{z \in V} \wnorm{M(z)} \le C_{M,V}.
  \end{equation*}
\end{definition}

For each $r > 0$ let $D_r = \{ z \in \C : \abs{z} < r \}$. Our first main result roughly speaking, says that one can `uniformly extend' the spectral stability of a family of operators satisfying (KL) to the corresponding twisted family of operators in some neighbourhood of $0$, provided the twist is compatible with $\wnorm{\cdot}$.

\begin{theorem}\label{theorem:convergence_of_eigendata_derivs}
  Let $\{ A_\epsilon \}_{\epsilon \ge 0}$ satisfy (KL), where $A_0$ is a simple quasi-compact operator with leading eigenvalue $\lambda_0$ satisfying $\alpha < \abs{\lambda_0}$, and let $M : D \to \C$ be a compactly $\wnorm{\cdot}$-bounded twist. Then there exists $\theta > 0$ such that for every compact $V \subseteq D_\theta$ there exists $\epsilon_V > 0$ and, for each $\epsilon \in [0, \epsilon_V]$, analytic functions\footnote{When we say a map is analytic on an arbitrary compact subset $V$ of $\C$, we mean that it may be extended to an analytic map on some larger open subset of $\C$.} $\lambda_\epsilon(\cdot): V \to \C$, $\Pi_\epsilon(\cdot): V \to L(B)$, and $N_\epsilon(\cdot) : V \to L(B)$ such that $A_\epsilon(z)$ is a simple quasi-compact operator with decomposition $A_\epsilon(z) = \lambda_\epsilon(z) \Pi_\epsilon(z) + N_\epsilon(z)$ whenever $z \in V$. Additionally, the derivatives of all orders of the twisted spectral data satisfy the following uniform H{\"o}lder estimate: there exists $\eta(V)$, and, for each $n \in \N$, a constant $O_n$ such that for all $z \in V$ and sufficiently small $\epsilon$ one has
  \begin{equation*}
    \max\left\{ \begin{array}{c} \abs{\lambda_\epsilon^{(n)}(z) - \lambda_\epsilon^{(n)}(0)}, \tnorm{\Pi_\epsilon^{(n)}(z) - \Pi_\epsilon^{(n)}(0)}, \\
    \tnorm{N_\epsilon^{(n)}(z) - N_\epsilon^{(n)}(0)} \end{array} \right\} \le O_n\tau(\epsilon)^{\eta(V)}.
  \end{equation*}
\end{theorem}

\begin{remark}
  We have only considered the case of stability of a leading simple `twisted' eigenvalue, as it is a substantial simplification and all that we require in applications. We anticipate that analogous results would hold, \emph{mutatis mutandis}, for any eigenvalue outside the essential spectrum of finite algebraic multiplicity.
\end{remark}

Let us describe the strategy for proving Theorem \ref{theorem:convergence_of_eigendata_derivs}.
Firstly, using the fact that $\{A_\epsilon\}_{\epsilon\ge 0}$ satisfies (KL), we show that there exists $\psi > 0$ such that $\{A_\epsilon(z)\}_{\epsilon \ge 0}$ satisfies (KL) uniformly in $z$ on compact subsets of $D_\psi$. In our setting standard arguments \cite{RE83,nagaev1957,kato1966perturbation,HennionHerve} imply that $A_0(z)$ is a simple quasi-compact operator on some $D_\theta$, where we may also assume that $\theta \in (0, \psi]$. Using a technical lemma concerning the boundedness of the resolvents of $A_0(z)$ on compact subsets of $D_\theta$, we then apply theory of \cite{keller1999stability} to obtain a uniform version of Proposition \ref{prop:keller_liverani_summary} for the family of operators $\{A_\epsilon(z)\}_{\epsilon \ge 0}$ when $z \in V$. Theorem \ref{theorem:convergence_of_eigendata_derivs} immediately follows.

\subsubsection{Proof of Theorem \ref{theorem:convergence_of_eigendata_derivs}}

\paragraph{Step 1: Verification of \ref{en:kl_conv}.}

\begin{lemma}\label{prop:twisted_KL1}
  Suppose $\{A_\epsilon\}_{\epsilon \ge 0}$ satisfies (KL), $M: D \to L(B)$ is a twist and $V \subseteq D$ is compact.
  Let $\tau_V:[0,\infty) \to [0, \infty)$ be defined by
  \begin{equation*}
    \tau_V(\epsilon) = \left(\sup_{z \in V} \norm{M(z)}\right) \tau(\epsilon).
  \end{equation*}
  Then $\tau_V$ is an upper-semicontinuous function, and \ref{en:kl_conv} holds for $\{A_\epsilon(z)\}_{\epsilon \ge 0}$ for every $z \in V$ with $\tau_V$ in place of $\tau$.
  \begin{proof}
    Note that $\tau_V$ is finite as $V$ is compact and $M$ is continuous on $V$.
    For each $\epsilon > 0$ and $z \in V$ the definition of $\tnorm{\cdot}$ implies that $\tnorm{A_\epsilon(z) - A_0(z)} \le \tnorm{A_\epsilon - A_0}\norm{M(z)}$, and so using \ref{en:kl_conv} we find that
    \begin{equation*}
      \sup_{z \in V} \tnorm{A_\epsilon(z) - A_0(z)} \le \tnorm{A_\epsilon - A_0}\left(\sup_{z \in V}\norm{M(z)}\right) \le \tau_V(\epsilon),
    \end{equation*}
    as required.
  \end{proof}
\end{lemma}

\paragraph{Step 2: Verification of \ref{en:kl_l1_bound}.}
\begin{lemma}\label{prop:twisted_KL2}
  If $\{A_\epsilon\}_{\epsilon \ge 0}$ satisfies (KL), $M: D \to L(B)$ is a compactly $\wnorm{\cdot}$-bounded twist, and $V \subseteq D$ is compact, then there exists $K_{1,V} > 0$ such that for every $\epsilon \ge 0$ and $n \in \N$ we have
  \begin{equation*}
    \sup_{z \in V} \wnorm{A_\epsilon(z)^n} \le K_{1,V}^n.
  \end{equation*}
  In particular, \ref{en:kl_l1_bound} holds for $\{A_\epsilon(z)\}_{\epsilon \ge 0}$ for every $z \in V$.
  \begin{proof}
    As $M$ is compactly $\wnorm{\cdot}$-bounded there exists $C_{M,V} > 0$ such that
    \begin{equation*}
      \sup_{z \in V} \wnorm{M(z)} \le C_{M,V}.
    \end{equation*}
    Set $K_{1,V} = C_1 K_1 C_{M,V}$. Then for each $z \in V$, $n \in N$ and $\epsilon \ge 0$ we have
    \begin{equation*}
      \wnorm{A_\epsilon(z)^n} \le \wnorm{A_\epsilon}^n\wnorm{M(z)}^n \le (C_1 K_1)^n C_{M,V}^n = K_{1,V}^n.
    \end{equation*}
  \end{proof}
\end{lemma}

\paragraph{Step 3: Verification of \ref{en:kl_ly_bound}.}

\begin{lemma}\label{prop:small_lasota_yorke}
  Under the hypotheses of Theorem \ref{theorem:convergence_of_eigendata_derivs} there exists $\psi > 0$, $\alpha_\psi \in (0,1)$ and $C_{2,\psi},C_{3,\psi},K_{2,\psi} > 0$ such that
  \begin{equation*}
    \norm{A_\epsilon(z)^n f} \le C_{2,\psi} \alpha_\psi^n \norm{f} + C_{3,\psi} K_{2,\psi}^n \wnorm{f}
  \end{equation*}
  for every $z \in D_{\psi}$, $f \in B$, and $n \in \N$.
  \begin{proof}
    Fix $m$ sufficiently large so that $C_2 \alpha^m < 1/2$. Using \ref{en:kl_ly_bound} for $\{A_\epsilon\}_{\epsilon \ge 0}$ yields
    \begin{equation}\begin{split}\label{eq:small_lasota_yorke_1}
      \norm{A_\epsilon(z)^m f} &\le \norm{A_\epsilon^m f} + \norm{A_\epsilon(z)^m - A_\epsilon^m}\norm{f} \\
      &\le \left(1/2 + \norm{A_\epsilon(z)^m - A_\epsilon^m}\right)\norm{f} + C_3 K^m \wnorm{f}.
    \end{split}\end{equation}
    As
    \begin{equation*}
      A_\epsilon(z)^m - A_\epsilon^m = \sum_{k=0}^{m-1} A_\epsilon(z)^{k} (A_\epsilon(z) - A_\epsilon) A_\epsilon(z)^{m-1-k},
    \end{equation*}
    we may apply \ref{en:kl_ly_bound} again to obtain
    \begin{equation}\begin{split}\label{eq:small_lasota_yorke_2}
      \lvert A_\epsilon(z)^m& - A_\epsilon^m\rvert \le \sum_{k=0}^{m-1} \norm{A_\epsilon}^{k+1} \norm{M(z) - M(0)} \norm{A_\epsilon(z)}^{m-1-k}\\
      &\le \norm{M(z) - M(0)} \sum_{k=0}^{m-1} (C_2 \alpha + C_3 K)^m \norm{M(z)}^{m-1-k}.
    \end{split}\end{equation}
    Since the right-hand side of \eqref{eq:small_lasota_yorke_2} is continuous in $z$ and vanishes at $z=0$, there exists $\psi > 0$ such that
    \begin{equation*}
      \sup_{z \in D_\psi} \sup_{\epsilon \ge 0} \norm{A_\epsilon(z)^m - A_\epsilon^m} \le 1/4.
    \end{equation*}
    Applying this to \eqref{eq:small_lasota_yorke_1}, for each $\epsilon \ge 0$ and $z \in D_\psi$ we have
    \begin{equation*}
      \norm{A_\epsilon(z)^m f} \le (3/4) \norm{f} + C_3 K^m\wnorm{f}.
    \end{equation*}
    We can use Lemma \ref{prop:twisted_KL2} to iterate this inequality, obtaining \ref{en:kl_ly_bound} for $\{A_\epsilon(z)^m\}_{\epsilon \ge 0}$ with uniform coefficients. Standard arguments imply that \ref{en:kl_ly_bound} also holds for $\{A_\epsilon(z)\}_{\epsilon \ge 0}$ with suitable modified coefficients.
  \end{proof}
\end{lemma}

We have now verified that under the hypotheses of Theorem \ref{theorem:convergence_of_eigendata_derivs} the families of operators $\{A_\epsilon(z)\}_{\epsilon \ge 0}$ satisfy (KL) uniformly in $z$ on every compact $V \subseteq D_\psi$.

\paragraph{Step 4: Quasi-compactness of $A_0(z)$.}

The following result is standard in the theory of analytic perturbations of linear operators \cite{RE83,nagaev1957,kato1966perturbation,HennionHerve}. We provide an outline of the proof in our setting.

\begin{lemma}\label{lemma:quasicompact_unperturbed_operator}
  Assume the hypotheses of Theorem \ref{theorem:convergence_of_eigendata_derivs}, and recall $\psi$ and $\alpha_\psi$ from Lemma \ref{prop:small_lasota_yorke}. There exists $\theta \in (0, \psi]$ and $\lambda_0(\cdot): D_{\theta} \to \C$, $\Pi_0(\cdot): D_{\theta} \to L(B)$ and $N_0(\cdot): D_{\theta} \to L(B)$ such that for each $z \in D_{\theta}$ the operator $A_0(z)$ has quasi-compact decomposition $A_0(z) = \lambda_0(z) \Pi_0(z) + N_0(z)$ and
  \begin{equation*}
    \max\left\{\alpha_\psi, \sup_{z \in D_\theta} \rho(N_0(z)) \right\} < \inf_{z \in D_{\theta}} \abs{\lambda_0(z)}.
  \end{equation*}
  \begin{proof}
    As $z \mapsto A_0(z)$ is analytic and $A_0$ is quasi-compact with decomposition $A_0 = \lambda_0 \Pi_0 + N_0$, it is standard that there exists $\theta > 0$, and analytic maps $\lambda_0(z): D_\theta \to \C$, $\Pi_0(z): D_\theta \to L(B)$ and $N_0(z): D_\theta \to L(B)$ such that $A_0(z)$ is a simple quasi-compact operator with decomposition $A_0(z) = \lambda_0(z) \Pi_0(z) + N_0(z)$ for each $z \in D_\theta$. By possibly shrinking $\theta$ we may assume that $\theta \le \psi$. Since $\lambda_0(\cdot)$ is analytic, we may shrink $\theta$ to guarantee that $\alpha_\psi < \inf_{z \in D_\theta} \abs{\lambda_0(z)}$. Furthermore, as the spectral radius is upper-semicontinuous as a function of the operator \cite[IV.3.1]{kato1966perturbation}, we may shrink $\theta$ more so that $\sup_{z \in D_\theta} \rho(N_0(z))< \inf_{z \in D_\theta} \abs{\lambda_0(z)}$.
  \end{proof}
\end{lemma}

\paragraph{Step 5: Uniform bounds on the norm of the resolvent of $A_0(z)$.}

In order to apply the theory in \cite{keller1999stability} `uniformly' to obtain Theorem \ref{theorem:convergence_of_eigendata_derivs}, we need a uniform bound for the norms of the resolvents of the twisted operators $A_0(z)$. For $A \in L(B)$, $\delta > 0$ and $r > \rho_{ess}(A)$ define
\begin{equation*}
  V_{\delta,r}(A) = \{ \omega \in \C : \abs{\omega} \le r \text{ or } \dist(\omega,\sigma(A)) \le \delta \}.
\end{equation*}
Noting that $\sigma(A) \subseteq V_{\delta,r}(A)$, let
\begin{equation*}
  J_{\delta,r}(A) = \sup \left\{ \norm{(\omega - A)^{-1}}: \omega \in \C \setminus V_{\delta, r}(A)\right\}.
\end{equation*}

\begin{lemma}\label{lemma:twisted_resolvent_bound}
  Assume the hypotheses of Theorem \ref{theorem:convergence_of_eigendata_derivs}, and recall $\theta$ from Lemma \ref{lemma:quasicompact_unperturbed_operator}.
  For every compact $V \subseteq D_\theta$, $\delta > 0$ and $r > \sup_{z \in D_\theta} \rho(N_0(z))$ we have
  \begin{equation*}
    \sup_{z \in V} J_{\delta,r}(A_0(z)) < \infty.
  \end{equation*}
  \begin{proof}
    For every $z \in D_\theta$ and $\omega \in \C\setminus \sigma(A_0(z))$ let $R(\omega,z) = (\omega - A_0(z))^{-1}$ denote the resolvent of $A_0(z)$ at $\omega$. Fix $z \in V$. Recall from Lemma \ref{lemma:quasicompact_unperturbed_operator} that $A_0(z)$ is a simple quasi-compact operator with decomposition $A_0(z) = \lambda_0(z)\Pi_0(z) + N_0(z)$. As $\lambda_0(z)$ is an isolated simple eigenvalue of $A_0(z)$, the partial-fraction decomposition of the resolvent \cite[III-(6.32)]{kato1966perturbation} yields
    \begin{equation}\label{eq:twisted_resolvent_bound_1}
      R(\omega, z) = (\omega - \lambda_0(z))^{-1} \Pi_0(z) + S(\omega,z),
    \end{equation}
    where for each $\omega \in \{\lambda_0(z)\} \cup \C\setminus \sigma(A_0(z))$ the operator $S(\omega, z) = \lim_{\omega' \to \omega} R(\omega',z)(\Id-\Pi_0(z))$ is the reduced resolvent of $A_0(z)$ with respect to $\lambda_0(z)$ at $\omega$ \cite[III-(6.30), III-(6.31)]{kato1966perturbation}. We have
    \begin{equation}\begin{split}\label{eq:twisted_resolvent_bound_2}
      \sup_{\omega \in \C \setminus V_{\delta, r}(A_0(z))} \norm{R(\omega, z)} \le &\delta^{-1} \sup_{z\in V}\norm{\Pi_0(z)} \\
      &+ \sup_{\omega \in \C \setminus V_{\delta, r}(A_0(z))} \norm{S(\omega,z)}.
    \end{split}\end{equation}
    Since $\Pi_0(\cdot)$ is analytic on $V$, which is compact, $\sup_{z\in V}\norm{\Pi_0(z)} < \infty$. Hence, to complete the Lemma it suffices to bound
    equation $\sup_{\omega \in \C \setminus V_{\delta, r}(A_0(z))} \norm{S(\omega,z)}$ uniformly in $z \in V$. As noted immediately after \cite[III-(6.25)]{kato1966perturbation}, when restricted to $(\Id - \Pi_0(z))B$ the operator $S(\omega, z)$ coincides with the resolvent of $A_0(z)(\Id - \Pi_0(z)) = N_0(z)$ at $\omega$. Since $S(\omega, z)$ vanishes on $\Pi_0(z)B$, this implies that $S(\omega, z) = (\omega - N_0(z))^{-1}(\Id - \Pi_0(z))$. As $r > \sup_{z \in D_\theta} \rho(N_0(z))$, for $\omega \in \C \setminus B_r(0)$ the standard Neumann series\footnote{While the formula in \cite[I-(5.10)]{kato1966perturbation} is only given for the finite-dimensional case, it also holds in the current setting by the same arguments.} for the resolvent \cite[I-(5.10)]{kato1966perturbation} yields
    \begin{equation}\label{eq:twisted_resolvent_bound_3}
      S(\omega, z) = \left(\sum_{k=0}^\infty \omega^{-k-1}N_0(z)^{k}\right)(\Id-\Pi_0(z)).
    \end{equation}

    By the spectral radius formula, there exists some $H > 0$ such that for every $n \in \N$ we have $\norm{N_0(z)^{n}} \le H \left(\sup_{z \in D_\theta} \rho(N_0(z))\right)^n$. Using \eqref{eq:twisted_resolvent_bound_3} we therefore have
    \begin{equation*}\begin{split}
      &\sup_{\omega \in \C \setminus V_{\delta, r}(A_0(z))} \norm{S(\omega, z)}
      \le \sup_{\omega \in \C \setminus B_r(0)} \norm{S(\omega, z)} \\
      & \le \sup_{\omega \in \C \setminus B_r(0)} \left(\sum_{k=0}^\infty \abs{\omega^{-k-1}}\norm{N_0(z)^{k}}\right)\norm{\Id-\Pi_0(z)} \\
      &\le \left(\sum_{k=0}^\infty \frac{H}{r}\left(r^{-1}\sup_{z \in D_\theta} \rho(N_0(z))\right)^k\right)\left( \sup_{z \in V} \norm{\Id-\Pi_0(z)}\right),
    \end{split}\end{equation*}
    which is finite as $\sup_{z \in D_\theta} \rho(N_0(z)) < r$ and $\Id-\Pi_0(z)$ is analytic on $V$. Recalling \eqref{eq:twisted_resolvent_bound_2} and the definition of $J_{\delta,r}(A_0(z))$, we complete the proof.
  \end{proof}
\end{lemma}

\paragraph{Step 6: Combining the previous steps.}
    Recall $\theta$ from Lemma \ref{lemma:quasicompact_unperturbed_operator} and let $V \subseteq D_\theta$ be compact.
    Out of necessity, we construct $\lambda_\epsilon(z), \Pi_\epsilon(z)$ and $N_\epsilon(z)$ for $z$ in a larger compact set whose interior contains $V$.
    As $V$ is compact there exists some $\gamma \in ( 0, \theta)$ such that $V$ is contained in $D_\gamma$. By Lemmas \ref{prop:twisted_KL1}, \ref{prop:twisted_KL2} and \ref{prop:small_lasota_yorke}, for each $z \in \clo{D_\gamma}$ the family of operators $\{A_\epsilon(z)\}_{\epsilon \ge 0}$ satisfies (KL) with data
    \begin{equation}\label{eq:convergence_of_eigendata_derivs_1}
      \tau_\gamma, C_{1,\gamma}, C_{2,\gamma}, C_{3,\gamma}, K_{1,\gamma}, K_{2,\gamma} \text{ and } \alpha_\gamma \in [0,1),
    \end{equation}
    only depending on $\clo{D_\gamma}$.

    By Lemma \ref{lemma:quasicompact_unperturbed_operator} there exists $r > \max\{ \alpha_\psi, \sup_{z \in D_\theta} \rho(N_0(z))\}$ and $\delta > 0$ such that $r + \delta < \inf_{z \in \clo{D_\gamma}} \abs{\lambda_0(z)}$, which implies that $B_{\delta}(\lambda_0(z)) \cap (B_r(0) \cup \sigma(A_0(z))) = \{\lambda_0(z) \}$ for every $ z\in \clo{D_\gamma}$.
    Hence, by \cite[Theorem 1 and the inequality (10)]{keller1999stability} for each $z \in \clo{D_\gamma}$ there exists $\epsilon_{\delta,r,z} > 0$ such that for $\epsilon \in [0, \epsilon_{\delta,r,z}]$ and $\omega \in \partial B_{\delta}(\lambda_0(z))$ the operator $(\omega - A_\epsilon(z))^{-1}$ is bounded and so the spectral projection
    \begin{equation}\label{eq:convergence_of_eigendata_derivs_2}
    \Pi_\epsilon(z) = \frac{1}{2 \pi i} \intf_{\partial B_{\delta}(\lambda_0(z))} (\omega - A_\epsilon(z))^{-1} d\omega
    \end{equation}
    is a well-defined element of $L(B)$.
    From the definitions of $\epsilon_0$ and $\epsilon_1$ in the proof of \cite[Corollary 1]{keller1999stability}, and the definition \cite[(13)]{keller1999stability}
    we see that $\epsilon_{\delta,r,z}$ may be chosen independently of $z \in \clo{D_\gamma}$ as (KL) is satisfied for each $\{A_\epsilon(z)\}_{\epsilon \ge 0}$ with data \eqref{eq:convergence_of_eigendata_derivs_1} independent of $z \in \clo{D_\gamma}$ and as
    \begin{equation}\label{eq:convergence_of_eigendata_derivs_4}
      \sup_{z \in \clo{D_\gamma}} J_{\delta,r}(A_0(z)) < \infty
    \end{equation}
    by Lemma \ref{lemma:twisted_resolvent_bound}.
    The same argument applied to \cite[part (3) of Corollary 1]{keller1999stability} implies that there exists $\epsilon_V \in (0, \epsilon_{\delta,r,z}]$ and $\delta_\gamma > 0$ such that if $\delta \in (0, \delta_\gamma)$ then $\rank(\Pi_\epsilon(z)) = \rank(\Pi_0(z)) =1$ for every $\epsilon \in [0, \epsilon_V]$ and $z \in \clo{D_\gamma}$.
    Since $\rank(\Pi_\epsilon(z)) = 1$, each $A_\epsilon(z)$ has a simple eigenvalue $\lambda_\epsilon(z) \in B_{\delta}(\lambda_0(z))$.
    By \cite[(10)]{keller1999stability} we have $\sigma(A_\epsilon(z)) \subseteq B_{\delta}(\lambda_0(z)) \cup B_{r}(0)$.
    Since $\sigma(A_\epsilon(z)) \cap B_{\delta}(\lambda_0(z)) = \{\lambda_\epsilon(z) \}$, it follows that $\sigma(A_\epsilon(z)) \setminus \{ \lambda_\epsilon(z) \} \subseteq B_{r}(0)$.
    Hence,
    \begin{equation}\label{eq:convergence_of_eigendata_derivs_3}
       \Id - \Pi_\epsilon(z) = \frac{1}{2 \pi i} \intf_{\partial B_{r}(0)} (\omega - A_\epsilon(z))^{-1} d\omega.
    \end{equation}
    Defining $N_\epsilon(z) = A_\epsilon(z)(\Id - \Pi_\epsilon(z))$, for every $z \in \clo{D_\gamma}$ and $\epsilon \in [0, \epsilon_V]$ we therefore have that $A_\epsilon(z)$ is a simple quasi-compact operator with decomposition $A_\epsilon(z) = \lambda_\epsilon(z) \Pi_\epsilon(z) + N_\epsilon(z)$.

    We will now show that for each $\epsilon \in [0, \epsilon_V]$ the maps $z \mapsto \lambda_\epsilon(z)$, $z \mapsto \Pi_\epsilon(z)$, and $z \mapsto N_\epsilon(z)$ are analytic on $D_\gamma$. As the contour in the integral in \eqref{eq:convergence_of_eigendata_derivs_3} is fixed, it is well-known that $z \mapsto \Id - \Pi_\epsilon(z)$ is analytic on $D_\gamma$ \cite[VIII.1.3 Theorem 1.7]{kato1966perturbation}.
    Hence $z \mapsto \Pi_\epsilon(z)$ is analytic on $D_\gamma$. As $z \mapsto A_\epsilon(z)$ is analytic and by the definition of $N_\epsilon(z)$, the map $z \mapsto N_\epsilon(z)$ is analytic on $D_\gamma$. Since $\lambda_\epsilon(z)$ has algebraic multiplicity 1, by the discussion in \cite[II.1.8]{kato1966perturbation}, the map $z \mapsto \lambda_\epsilon(z)$ is analytic on $D_\gamma$. Hence the maps $\lambda_\epsilon(\cdot), \Pi_\epsilon(\cdot)$ and $N_\epsilon(\cdot)$ are analytic on $V$ as they may be extended to analytic maps on an open subset of $\C$ that contains $V$, namely $D_\gamma$.

    We now confirm that the required H{\"o}lder estimate holds for the various spectral data using (KL) for $\{A_\epsilon(z)\}_{\epsilon \ge 0}$ and with uniform data as in \eqref{eq:convergence_of_eigendata_derivs_2}.
    By \cite[Corollary 1]{keller1999stability}, there exists $H_{\delta,r,z} > 0$ such that $\tnorm{\Pi_{\epsilon}(z) - \Pi_{0}(z)} \le H_{\delta,r,z} \tau_V(\epsilon)^{\eta(V)}$ for every $\epsilon \in [0 , \epsilon_V]$ and $z \in \clo{D_\gamma}$, where
    \begin{equation*}
      \eta(V) = \frac{\log( r / \alpha_\gamma) }{\log(\max\{K_{1,\gamma}, K_{2,\gamma}\} / \alpha_\gamma) }.
    \end{equation*}
    Recalling the bound \eqref{eq:convergence_of_eigendata_derivs_4} and that \eqref{eq:convergence_of_eigendata_derivs_1} is independent of $\clo{D_\gamma}$, we conclude from the proof of \cite[Corollary 1]{keller1999stability} that $H_{\delta,r,z}$ can be chosen independently of $z \in \clo{D_\gamma}$.
    Moreover, by Lemma \ref{prop:twisted_KL1} we have $\tau_V(\epsilon) = \sup_{z \in \clo{D_\gamma}} \wnorm{M(z)} \tau(\epsilon)$. Hence, if we set $O_0' = H_{\delta,r,z}\sup_{z \in \clo{D_\gamma}} \wnorm{M(z)}^{\eta(V)}$ then for all $\epsilon \in [0 , \epsilon_V]$ we have
    \begin{equation}\label{eq:convergence_of_eigendata_derivs_5}
      \sup_{z \in \clo{D_\gamma}} \tnorm{\Pi_{\epsilon}(z) - \Pi_{0}(z)} \le O_0' \tau(\epsilon)^{\eta(V)}.
    \end{equation}
    By definition we have
    \begin{equation*}
      (\lambda_0(z) - \lambda_\epsilon(z))\Pi_0(z) = (\lambda_\epsilon(z)-A_\epsilon(z))(\Pi_\epsilon(z) - \Pi_0(z)) + (A_0(z) - A_\epsilon(z))\Pi_0(z),
    \end{equation*}
    and so
    \begin{equation}\begin{split}\label{eq:convergence_of_eigendata_derivs_6}
      \abs{\lambda_0(z) - \lambda_\epsilon(z)}\tnorm{\Pi_0(z)} \le & (\abs{\lambda_\epsilon(z)}+\wnorm{A_\epsilon(z)})\tnorm{\Pi_\epsilon(z) - \Pi_0(z)} \\
      &+ \tnorm{A_0(z) - A_\epsilon(z)}\norm{\Pi_0(z)}.
    \end{split}\end{equation}
    For every $\epsilon \in [0, \epsilon_V]$ and $z \in \clo{D_\gamma}$ we have  $\wnorm{A_\epsilon(z)} \le K_{1,\gamma}$ and $\abs{\lambda_\epsilon(z)} \le \abs{\lambda_0(0)} + \delta$.
    Provided that $\gamma$ is sufficiently small, which we may guarantee by shrinking $\theta$, by the analyticity of $z \mapsto \Pi_0(z)$ we have that $\sup_{z \in \clo{D_\gamma}} \norm{\Pi_0(z)} < \infty$ and
    \begin{equation*}
      \inf_{z \in \clo{D_\gamma}} \tnorm{\Pi_0(z)} \ge \tnorm{\Pi_0(0)} - \inf_{z \in \clo{D_\gamma}} \norm{\Pi_0(z) - \Pi_0(0)} > 0.
    \end{equation*}
    By the estimates in the previous two sentences, Lemma \ref{prop:twisted_KL1} and \eqref{eq:convergence_of_eigendata_derivs_5} it follows that for every $z \in \clo{D_\gamma}$ and $\epsilon \in [0, \epsilon_V]$ we have
    \begin{equation*}\begin{split}
      \bigg(&\frac{\abs{\lambda_\epsilon(z)}+\wnorm{A_\epsilon(z)}}{\tnorm{\Pi_0(z)}}
      \bigg)\tnorm{\Pi_\epsilon(z) - \Pi_0(z)}
      + \frac{\norm{\Pi_0(z)}}{\tnorm{\Pi_0(z)}}\tnorm{A_0(z) - A_\epsilon(z)} \\
      &\le \left(\frac{(K_{1,\gamma} + \delta + \abs{\lambda_0(0)})O_0' + \sup_{z \in \clo{D_\gamma}} \norm{\Pi_0(z)} }{\inf_{z \in \clo{D_\gamma}} \tnorm{\Pi_0(z)}} \right)\tau(\epsilon)^{\eta(V)}  \\
      &:= O_0''\tau(\epsilon)^{\eta(V)} < \infty,
    \end{split}\end{equation*}
    which, when applied to \eqref{eq:convergence_of_eigendata_derivs_6}, yields
    \begin{equation*}
      \sup_{z \in \clo{D_\gamma}} \abs{\lambda_0(z) - \lambda_\epsilon(z)} \le O_0'' \tau(\epsilon)^{\eta(V)}.
    \end{equation*}
    Examining the proof of \cite[Corollary 2]{keller1999stability} and using the same arguments as before, we similarly find a constant $O_0'''$ such that
    \begin{equation*}
      \sup_{z \in \clo{D_\gamma}} \tnorm{ N_\epsilon(z) - N_0(z)} \le O_0''' \tau(\epsilon)^{\eta(V)}.
    \end{equation*}
    Since $V \subseteq \clo{D_\gamma}$, the required uniform H{\"o}lder estimate for the undifferentiated spectral data holds on $V$ with $O_0 = \max\{O_0',O_0'', O_0'''\}$ (i.e. we obtain the conclusion of Theorem \ref{theorem:convergence_of_eigendata_derivs} in the case where $n = 0$).
    For every compact subset of $D_\gamma$ one derives a uniform H{\"o}lder estimate for the $n^{th}$ derivative of  $\lambda_\epsilon(\cdot)$, $\Pi_\epsilon(\cdot)$ and $N_\epsilon(\cdot)$ by a standard application of Cauchy's integral formula along the contour $\partial D_\gamma$.
    In particular, we obtain the required uniform H{\"o}lder estimate on $V$, which concludes the proof of Theorem \ref{theorem:convergence_of_eigendata_derivs}.

\begin{remark}
  Note that $0 < \inf_{z \in V, \epsilon \in [0, \epsilon_V]} \abs{\lambda_\epsilon(z)}$. This bound will be important when defining approximate rate functions in Section \ref{subsec:stability_estimation_of_rate_function}.
\end{remark}

\section{Stability of statistical limit laws}\label{sec:application_to_NG}

  Let $(X,m)$ be a probability space, $T: X \to X$ a non-singular transformation, and $\LL$ the Perron-Frobenius operator associated with $T$.
  For a comprehensive overview of the Nagaev-Guivarc'h spectral method for obtaining statistical laws for dynamical systems we refer the reader to \cite{gouezel2015limit}, while for a broader description of both the method and its history we refer to \cite{HennionHerve, aimino2015note, RE83, GH88, nagaev1957}.
  In Section \ref{subsec:stability_estimation_of_variance} we show that if $T$ satisfies a central limit theorem (CLT) due to the Nagaev-Guivarc'h method then the variance of the CLT is stable with respect to perturbations of $\LL$ satisfying (KL), and in Section \ref{subsec:stability_estimation_of_rate_function} we obtain the analogous result for the stability of the rate function when $T$ satisfies a large deviation principle (LDP).
  We now detail a version of the method \cite{HennionHerve, aimino2015note}, incorporating some conditions from \cite{keller1999stability}, which we require for our stability results.

\begin{definition}\label{def:strong_space_condition}
 A Banach space $(B, \norm{\cdot})$ contained in $L^1(m)$ satisfies condition $(S)$ if each of the following conditions holds.
 \begin{enumerate}[label=(S\arabic*)]
   \item $\lnorm{\cdot} \le \norm{\cdot}$, where $\lnorm{\cdot}$ is the $L^1(m)$ norm.
   \item The closed unit ball of $(B, \norm{\cdot})$ is compact in $(L^1(m), \lnorm{\cdot})$.
   \item The constant functions are in $B$.
   \item $B$ is a Banach algebra i.e. there exists $C_B > 0$ such that $\norm{fg} \le C_B \norm{f}\norm{g}$ for all $f,g \in B$.
   \item $B$ is a vector sublattice of $L^1(m)$ i.e. $\conj{f}, \abs{f} \in B$ for every $f \in B$.
 \end{enumerate}
\end{definition}

Under (S) and some conditions on the Perron-Frobenius operator, such as quasi-compactness, the Nagaev-Guivarc'h method yields a CLT and LDP for $T$.

\begin{theorem}[{\cite[Theorems 1 and 2]{aimino2015note}}]
\label{aiminothm}
For each $n \in \Z^+$ and $g \in B$ we denote the partial Birkhoff sums $\sum_{k=0}^{n-1} g \circ T^k$ by $S_n(g)$.
If $(B, \norm{\cdot})$ satisfies \hyperref[def:strong_space_condition]{$(S)$} and the Perron-Frobenius operator $\LL:B\to B$ is simple and quasi-compact with $\rho(\LL) = 1$, then there exists a unique $T$-invariant measure $\mu$ such that $\frac{\mathrm{d}\mu}{\mathrm{d}m} \in B$ and
for any real-valued $g \in B \cap L^\infty(m)$ satisfying $\intf g d\mu = 0$ we have
\begin{enumerate}
\item
The variance $\lim_{n \to \infty} \intf \frac{S_n(g)^2}{n} d\mu =: \sigma_g^2$ exists.
\item
The sequence of random variables $\frac{S_n(g)}{\sqrt{n}}$ converges in distribution as $n \to \infty$ to a $N(0, \sigma_g^2)$ random variable in the probability space $(X,\nu)$ for every probability measure $\nu$ such that $\frac{\mathrm{d}\nu}{\mathrm{d}m} \in B$.
\item If $\sigma^2_g > 0$, there exists an open interval $I \subseteq \R$ containing $0$ and a rate function $r_g : I \to \R$, which is non-negative, continuous, strictly convex, and with unique minimum $r_g(0) = 0$, such that
  \begin{equation*}
    \lim_{n \to \infty} \frac{1}{n}\ln\nu( S_n(g) > ns ) = -r_g(s)
  \end{equation*}
  for every $s \in I \cap (0,\infty)$ and probability measure $\nu$ such that $\frac{\mathrm{d}\nu}{\mathrm{d}m} \in B$.
\end{enumerate}
\end{theorem}

Before stating our stability results we describe how the variance and rate function are determined.
Assume the setting of Theorem \ref{aiminothm}. Let $M_g : \C \to L(B)$ be defined by $M_g(z)(f) = e^{zg}f$, where $e^{zg}$ is defined by the usual power series. From this power-series representation it follows that $M_g$ is a $\lnorm{\cdot}$-bounded twist. The twisted Perron-Frobenius operator is then defined by $\LL(z) := \LL M_g(z)$. As $\LL$ is a simple quasi-compact operator, Lemma  \ref{lemma:quasicompact_unperturbed_operator} implies that $\LL(z)$ has quasi-compact decomposition
\begin{equation*}
  \LL(z) = \lambda(z)\Pi(z) + N(z),
\end{equation*}
where $\lambda(\cdot), \Pi(\cdot)$ and $N(\cdot)$ are analytic functions on some neighbourhood of $0$, say $D_\theta$.
The variance of the CLT \cite[Lemma IV.3]{HennionHerve} is given by
\begin{equation}\label{eq:variance_2nd_deriv_twisted_eigenvalue}
  \sigma^2_g = \lambda^{(2)}(0),
\end{equation}
and the rate function of the LDP \cite[Proposition VIII.3]{HennionHerve} is
\begin{equation}\label{eq:rate_function_twisted_eigenvalue}
  r_g(s) = \sup_{z \in (-\theta,\theta)} (sz - \ln(\lambda(z))),
\end{equation}
or, equivalently, the convex conjugate (Definition \ref{def:convex_conjugate}) of the function $\Lambda : (-\theta, \theta) \to \R$ defined by $\Lambda(z) = \log(\lambda(z))$.
As both the variance and the rate function depend on the derivatives of the spectral data of the twisted Perron-Frobenius operator, we can use Theorem \ref{theorem:convergence_of_eigendata_derivs} to prove their stability.

\begin{corollary}\label{thm:application_to_ng_method}
  Let $(B,\norm{\cdot})$, $\mathcal{L}$, and $g$ satisfy the hypotheses of Theorem \ref{aiminothm}, and define the $\lnorm{\cdot}$-compactly bounded twist $M_g : \C \to L(B)$ by $M_g(z)(f) = e^{zg}f$.
  Let $\{\LL_\epsilon\}_{\epsilon \ge 0 } \subseteq L(B)$ be a family of operators satisfying (KL), with $\LL_0 = \LL$. There exists $\theta > 0$, for each compact $V \subseteq D_\theta$, an $\epsilon_V > 0$ and, for each $\epsilon \in [0, \epsilon_V]$, analytic maps $\lambda_\epsilon : V \to \C$,  $\Pi_\epsilon : V \to L(B)$, and $N_\epsilon : V \to L(B)$, such that for each $z \in V$ the operator $\LL_\epsilon(z)$ is quasi-compact with decomposition $\LL_\epsilon(z) = \lambda_\epsilon(z) \Pi_\epsilon(z) + N_\epsilon(z)$. Moreover, as $\epsilon \to 0$ the spectral data $\lambda_\epsilon(\cdot),\Pi_\epsilon( \cdot)$ and $N_\epsilon(\cdot)$ converge in the manner described in Theorem \ref{theorem:convergence_of_eigendata_derivs}.
  \end{corollary}
  \begin{proof}
    Note that condition $(S)$ implies that $(B, \norm{\cdot})$ satisfies the requirements on the Banach space in Section \ref{sec:abstract_twist_stability}, with $\wnorm{\cdot}$ taken to be $\lnorm{\cdot}$. As noted, $M_g$ is a compactly $\lnorm{\cdot}$-bounded twist. Since $\LL$ is a simple quasi-compact operator with $\rho(\LL) = 1$ it has leading eigenvalue 1, which is strictly greater than the constant $\alpha$ in \ref{en:kl_ly_bound}. Thus we may apply Theorem  \ref{theorem:convergence_of_eigendata_derivs}, which yields the required conclusion.
  \end{proof}

  \subsection{Stability and estimation of the variance}\label{subsec:stability_estimation_of_variance}

  In view of \eqref{eq:variance_2nd_deriv_twisted_eigenvalue}, Corollary \ref{thm:application_to_ng_method} immediately yields the stability of the variance.

  \begin{theorem}\label{thm:stability_of_variance}
    Under the hypotheses of Corollary \ref{thm:application_to_ng_method}, $\lambda_\epsilon^{(2)}(0)$ is well defined for sufficiently small $\epsilon$ and the map $\epsilon \mapsto \lambda_\epsilon^{(2)}(0)$ is H{\"o}lder at $\epsilon = 0$. In particular we have $\lim_{\epsilon \to 0} \lambda_\epsilon^{(2)}(0) = \sigma^2_g$.
    \begin{proof}
      Let $V$ be a compact neighbourhood of $0$ with non-empty interior that is contained in $D_\theta$, where $\theta$ is from Theorem \ref{theorem:convergence_of_eigendata_derivs}.
      By Theorem \ref{theorem:convergence_of_eigendata_derivs} there exists $\epsilon_V > 0$ and, for each $\epsilon \in [0, \epsilon_V]$, analytic maps $\lambda_\epsilon(\cdot): V \to \C$ such that $\lambda_\epsilon(z)$ is a simple leading eigenvalue of $\LL_\epsilon(z)$ and the maps $\epsilon \mapsto \lambda_\epsilon^{(2)}(z)$ are uniformly H{\"o}lder at $\epsilon = 0$ for $z \in V$.
      In particular, the map $\epsilon \mapsto \lambda_\epsilon^{(2)}(0)$ is H{\"o}lder at $\epsilon = 0$ and so $\lim_{\epsilon \to 0} \lambda_\epsilon^{(2)}(0) = \lambda_0^{(2)}(0) = \sigma^2_g$.
    \end{proof}
  \end{theorem}

  In Sections \ref{sec:piecewise_expanding_one_dim} and \ref{sec:piecewise_expanding_multi_dim} we will consider the case where $\LL_\epsilon$ is a numerical approximation of $\LL$.
  In this context Theorem \ref{thm:stability_of_variance} provides a method for rigorously approximating $\sigma^2_g$ by computing $\lambda_\epsilon^{(2)}(0)$.
  Our proposed numerical methods will exploit the analyticity of $M_g(z)$ and the fact that each $\mathcal{L}_\epsilon(z)$ is a simple quasi-compact operator by using an explicit expression for $\lambda^{(2)}(z)$ in terms of the 0th order spectral data of $\LL_\epsilon(z)$, evaluated at $z=0$.

  \begin{proposition}\label{prop:2nd_deriv_twisted_eigenvalue_formula}
    Under the hypotheses of Theorem \ref{thm:application_to_ng_method}, if for any $z \in D_{\theta}$ and sufficiently small $\epsilon \ge 0$ we choose $v_{\epsilon,z} \in B$ and $\varphi_{\epsilon,z} \in B^*$ so that $\Pi_\epsilon(z)f = \varphi_{\epsilon,z}(f)v_{\epsilon,z}$, then
    \begin{equation}\begin{split}
    \label{2ndderiv}
      \lambda_\epsilon^{(2)}(z) &= \lambda_{\epsilon}(z)\varphi_{\epsilon,z}(g^2v_{\epsilon,z})\\
      &+ 2 \lambda_{\epsilon}(z)\varphi_{\epsilon,z}(g
      (\lambda_{\epsilon}(z) - \LL_{\epsilon}(z))^{-1}\LL_{\epsilon}(z)(\Id- \Pi_{\epsilon}(z))(gv_{\epsilon,z})).
    \end{split}\end{equation}
    \begin{proof}
      For any $z \in D_\theta$ there exists a compact set $V$ such that $z$ is in the interior of $V$. Assume $\epsilon \in [0, \epsilon_V]$.
      Differentiating the identity $(\lambda_\epsilon(z) - \LL_\epsilon(z))\Pi_\epsilon(z) = 0$ once with respect to $z$ yields
      \begin{equation}\label{eq:2nd_deriv_twisted_eigenvalue_formula_1}
        \lambda_\epsilon^{(1)}(z)\Pi_\epsilon(z) =  \LL_\epsilon^{(1)}(z)\Pi_\epsilon(z) - (\lambda_\epsilon(z) - \LL_\epsilon(z))\Pi_\epsilon^{(1)}(z),
      \end{equation}
      while differentiating a second time yields
      \begin{equation}\begin{split}
          \lambda_\epsilon^{(2)}(z)\Pi_\epsilon(z) =   &\LL_\epsilon^{(2)}(z)\Pi_\epsilon(z) - 2(\lambda_\epsilon^{(1)}(z) - \LL_\epsilon^{(1)}(z))\Pi_\epsilon^{(1)}(z) \\
          &- (\lambda_\epsilon(z)
          - \LL_\epsilon(z))\Pi_\epsilon^{(2)}(z).
      \end{split}\end{equation}
      As $\Pi_\epsilon(z)(\lambda_\epsilon(z) - \LL_\epsilon(z)) = 0$, by applying $\Pi_{\epsilon}(z)$ on the left of \eqref{eq:2nd_deriv_twisted_eigenvalue_formula_1} we obtain
      \begin{equation}\label{eq:2nd_deriv_twisted_eigenvalue_formula_2}
        \lambda_\epsilon^{(1)}(z)\Pi_{\epsilon}(z) = \Pi_{\epsilon}(z) \LL_\epsilon^{(1)}(z)\Pi_\epsilon(z).
      \end{equation}
      Similarly,
      \begin{equation}\label{eq:2nd_deriv_twisted_eigenvalue_formula_3}
        \lambda_\epsilon^{(2)}(z)\Pi_\epsilon(z) =   \Pi_\epsilon(z)\LL_\epsilon^{(2)}(z)\Pi_\epsilon(z) - 2\Pi_\epsilon(z)(\lambda_\epsilon^{(1)}(z) - \LL_\epsilon^{(1)}(z))\Pi_\epsilon^{(1)}(z).
      \end{equation}
      As $\lambda_{\epsilon}(z)$ is an isolated simple eigenvalue, by \cite[II-(2.14)]{kato1966perturbation} we have\footnote{We note that the sign discrepancy between \cite[II-(2.14)]{kato1966perturbation} and \eqref{eq:2nd_deriv_twisted_eigenvalue_formula_4} is due to an additional factor of $-1$ in the definition of the resolvent in \cite{kato1966perturbation}.}
      \begin{equation}\label{eq:2nd_deriv_twisted_eigenvalue_formula_4}
        \Pi_{\epsilon}^{(1)}(z) = \Pi_{\epsilon}(z)\LL^{(1)}_{\epsilon}(z)S_\epsilon(z) +
        S_\epsilon(z)\LL^{(1)}_{\epsilon}(z) \Pi_{\epsilon}(z),
      \end{equation}
      where $S_{\epsilon}(z) = (\lambda_{\epsilon}(z) - \LL_{\epsilon}(z))^{-1}(\Id- \Pi_{\epsilon}(z))$.
      Note that $\Pi_\epsilon(z)S_{\epsilon}(z) = S_{\epsilon}(z) \Pi_\epsilon(z) = 0$. Applying \eqref{eq:2nd_deriv_twisted_eigenvalue_formula_4} to \eqref{eq:2nd_deriv_twisted_eigenvalue_formula_3}, we find that
      \begin{equation*}\begin{split}
          \lambda_\epsilon^{(2)}(z)&\Pi_\epsilon(z) \\
          =   &\Pi_\epsilon(z)\LL_\epsilon^{(2)}(z)\Pi_\epsilon(z) - 2 \lambda_{\epsilon}^{(1)}(z)\Pi_{\epsilon}(z)\Pi_{\epsilon}^{(1)}(z) \\
          &+ 2 \Pi_\epsilon(z)\LL_\epsilon^{(1)}(z)\Pi_{\epsilon}^{(1)}(z) \\
          = &\Pi_\epsilon(z)\LL_\epsilon^{(2)}(z)\Pi_\epsilon(z) - 2 \lambda_{\epsilon}^{(1)}(z)\Pi_{\epsilon}(z)\LL^{(1)}_{\epsilon}(z)S_\epsilon(z) \\
          &+ 2 \Pi_\epsilon(z)\LL_\epsilon^{(1)}(z)(\Pi_{\epsilon}(z)\LL^{(1)}_{\epsilon}(z)S_\epsilon(z) +
          S_\epsilon(z)\LL^{(1)}_{\epsilon}(z) \Pi_{\epsilon}(z)).
      \end{split}\end{equation*}
      Applying $\Pi_\epsilon(z)$ on the right then yields
      \begin{equation}\begin{split}\label{eq:2nd_deriv_twisted_eigenvalue_formula_5}
        \lambda_\epsilon^{(2)}(z)\Pi_\epsilon(z)
        = &\Pi_\epsilon(z)\LL_\epsilon^{(2)}(z)\Pi_\epsilon(z) \\
        &+ 2 \Pi_\epsilon(z)\LL_\epsilon^{(1)}(z)
        S_\epsilon(z)\LL^{(1)}_{\epsilon}(z) \Pi_{\epsilon}(z).
      \end{split}\end{equation}
      Recall that the $\LL_\epsilon(z) = \LL_\epsilon M_g(z)$, where $M_g(z)(f) = e^{zg}f$. For each $n \in N$ and $f \in B$ we therefore have
      \begin{equation}\label{eq:2nd_deriv_twisted_eigenvalue_formula_6}
        \LL_\epsilon^{(n)}(z)(f) = \LL_\epsilon M_g^{(n)}(z)(f) = \LL_\epsilon (g^n e^{zg}f) = \LL_\epsilon(z)(g^n f).
      \end{equation}
      As $\Pi_\epsilon(z)(f) = \varphi_{\epsilon,z}(f)v_{\epsilon,z}$, the left and right eigenvectors of $\LL_\epsilon(z)$ for the eigenvalue $\lambda_\epsilon(z)$ are $v_{\epsilon,z}$ and $\varphi_{\epsilon,z}$, respectively. Moreover, as $\Pi_\epsilon(z)$ is a projection, we have $\varphi_{\epsilon,z} \Pi_\epsilon(z) = \varphi_{\epsilon,z}$ and $\Pi_\epsilon(z)v_{\epsilon,z} = v_{\epsilon,z}$.
      Using \eqref{eq:2nd_deriv_twisted_eigenvalue_formula_6}, and then applying $\varphi_{\epsilon,z}$ on the left and $v_{\epsilon,z}$ on the right to \eqref{eq:2nd_deriv_twisted_eigenvalue_formula_5}, we obtain
      \begin{equation*}\begin{split}
        \lambda_\epsilon^{(2)}(z) &= \lambda_{\epsilon}(z)\varphi_{\epsilon,z}(g^2v_{\epsilon,z})
        \\
        &+ 2 \lambda_{\epsilon}(z)\varphi_{\epsilon,z}(g
        (\lambda_{\epsilon}(z) - \LL_{\epsilon}(z))^{-1}(\Id- \Pi_{\epsilon}(z))\LL_{\epsilon}(z)(gv_{\epsilon,z})).
      \end{split}\end{equation*}
      We obtain the required statement upon noting that $\Id- \Pi_{\epsilon}(z)$ and $\LL_{\epsilon}(z)$ commute.
    \end{proof}
  \end{proposition}

  We now simplify the formula for $\lambda^{(2)}(z)$ in Proposition \ref{prop:2nd_deriv_twisted_eigenvalue_formula} in the case $z=0$ and assuming that $\LL_\epsilon$ is a Markov operator.
  Evaluating \eqref{2ndderiv} at $z = 0$ yields
  \begin{equation}\begin{split}\label{eq:variance_for_comp_1}
    \lambda_\epsilon^{(2)}(0) = &\lambda_{\epsilon}(0)\varphi_{\epsilon,0}(g^2v_{\epsilon,0})\\
    &+ 2 \lambda_{\epsilon}(0)\varphi_{\epsilon,0}(g
    (\lambda_{\epsilon}(0) - \LL_{\epsilon})^{-1}\LL_{\epsilon}(\Id- \Pi_{\epsilon})(gv_{\epsilon,0})).
  \end{split}\end{equation}
  As $\LL_\epsilon$ is a Markov operator, we have $\lambda_{\epsilon}(0) = 1$. Additionally, we may take $\varphi_{\epsilon,0}$ to be the map $f \mapsto \intf f dm$, which implies that $\intf v_{\epsilon,0} dm = \varphi_{\epsilon,0}(v_{\epsilon,0}) = 1$. When applied to \eqref{eq:variance_for_comp_1} this yields
  \begin{equation*}
    \lambda_\epsilon^{(2)}(0) = \intf g^2v_{\epsilon,0}
    + 2g
    (\Id - \LL_{\epsilon})^{-1}\LL_{\epsilon}(\Id- \Pi_{\epsilon})(gv_{\epsilon,0}) dm.
  \end{equation*}
  We may replace $(\Id- \Pi_{\epsilon})(gv_{\epsilon,0})$ with $gv_{\epsilon,0}$, which is equivalent to setting $\intf gv_{\epsilon,0} dm =0$, to obtain
  \begin{equation}\label{eq:variance_for_comp_2}
    \lambda_\epsilon^{(2)}(0) = \intf g^2v_{\epsilon,0}
    + 2g
    (\Id - \LL_{\epsilon})^{-1}\LL_{\epsilon}(gv_{\epsilon,0}) dm,
  \end{equation}
  which is the approximation used in the computation of the variance.
  When $\epsilon = 0$ the expression for \eqref{eq:variance_for_comp_2} is equal to the expression for $\lambda^{(2)}(0)$ from \cite[Corollary III.11]{HennionHerve}, which contains an alternative derivation. Proposition \ref{prop:2nd_deriv_twisted_eigenvalue_formula} provides a more general derivative-free expression for $\lambda^{(2)}_\epsilon(z)$ for any $z \in D_\theta$ and small $\epsilon$, in the case where $\LL_\epsilon(z)$ need not be a Markov operator.

  \begin{remark}\label{remark:stability_of_variance}
    The expression (\ref{eq:variance_for_comp_2}) provides an alternative approach for proving the stability of the variance, which has been exploited previously (e.g. in \cite{gouezel2006banach}):
    each of the terms on the right hand side of (\ref{eq:variance_for_comp_2}) may be approximated using the results from \cite{keller1999stability}.
   In contrast, our proof for the stability of the rate function in the next section requires uniform control of $\lambda_\epsilon(z)$ for $z$ in a real neighbourhood of $0$, for which the theory developed in \cite{keller1999stability} is insufficient and, to the best knowledge of the authors, a result such as Theorem \ref{theorem:convergence_of_eigendata_derivs} is required.
  \end{remark}

  \subsection{Stability and estimation of the rate function}
  \label{subsec:stability_estimation_of_rate_function}

  We begin with the definition of the convex conjugation, which is also known as the Legendre-Fenchel transform.
  \begin{definition}\label{def:convex_conjugate}
    The \emph{convex conjugate} of a function $f : I \to \R$, where $I \subseteq \R$ is an interval, is the function $f^* : \R \to \R$ defined by
    \begin{equation*}
      f^*(y) = \sup_{x \in I} \, (x y - f(x)).
    \end{equation*}
  \end{definition}

  Recall $\theta$ from Corollary \ref{thm:application_to_ng_method} and let $V$ be a closed interval contained in $(-\theta, \theta)$. For the purpose of proving stability of the rate function we assume that each $\LL_\epsilon$ is a positive operator. Fix $\epsilon \in [0, \epsilon_V]$ and $z \in V$. As $g$ is real valued, $\LL_\epsilon(z)$ is therefore also a positive operator.
  It follows that $\rho(\LL_\epsilon(z)) = \lambda_\epsilon(z)$ is non-negative \cite[Proposition V.4.1]{SchaeferHelmutH1974Blap}.
  In view of the remark at the end of Section \ref{sec:abstract_twist_stability}, we have $\abs{\lambda_\epsilon(z)} > 0$ and so $\lambda_\epsilon(z)$ is positive.

  For each $\epsilon \in [0, \epsilon_V]$ let $\Lambda_{\epsilon,V}: V \to \R$ be defined by $\Lambda_{\epsilon,V}(z) = \ln(\lambda_{\epsilon}(z))$, and note that
  \begin{equation*}
    \Lambda_{\epsilon,V}^*(s) = \sup_{z \in V}  \, (sz - \ln(\lambda_{\epsilon}(z))).
  \end{equation*}
  Compare this with the rate function $r_g$ for the LDP (Theorem \ref{aiminothm}), which is defined in \eqref{eq:rate_function_twisted_eigenvalue} to be the convex conjugate of the function $\ln(\lambda_0(\cdot)): (-\theta, \theta) \to \R$.
  Our main result for this section is the following.

  \begin{theorem}\label{thm:convergence_of_rate_functions}
    Suppose that we are in the setting of Corollary \ref{thm:application_to_ng_method}, with the additional requirement that each $\LL_\epsilon$ is a positive operator.
    For every closed interval $V \subseteq (-\theta,\theta)$ we have the following convergence of $\Lambda_{\epsilon,V}^*$ to $\Lambda_{0,V}^*$ as $\epsilon \to 0$: with $\eta(V)$ as supplied by Theorem \ref{theorem:convergence_of_eigendata_derivs} there exists $C_V$ such that for $\epsilon \in [0, \epsilon_V]$ one has
    \begin{equation*}
      \sup_{y \in \R} \abs{\Lambda_{\epsilon,V}^*(y) - \Lambda_{0,V}^*(y)} \le C_V \tau(\epsilon)^{\eta(V)}.
    \end{equation*}
    \begin{proof}
      By Corollary \ref{thm:application_to_ng_method} for every $\epsilon \in [0,\epsilon_V]$ we have
      \begin{equation*}
        \sup_{z \in V} \abs{\lambda_\epsilon(z) - \lambda_0(z)} \le O_0 \tau(\epsilon)^{\eta(V)}.
      \end{equation*}
      Moreover, $\inf_{z \in V, \epsilon \in [0,\epsilon_V]} \lambda_\epsilon(z) > 0$ by the remark at the end of Section \ref{sec:abstract_twist_stability}.
      Hence $z \mapsto \Lambda_{\epsilon,V}(z)$ is well-defined for every $\epsilon \in [0,\epsilon_V]$ and
      \begin{equation*}
        \abs{\Lambda_{0,V}(z) - \Lambda_{\epsilon,V}(z)} \le \left(\inf_{z \in V, \epsilon \in [0,\epsilon_V]} \lambda_\epsilon(z)\right)^{-1} O_0 \tau(\epsilon)^{\eta(V)} := C_V  \tau(\epsilon)^{\eta(V)}.
      \end{equation*}

      Let $y \in \R$. Since $V$ is compact and $z \mapsto zy - \Lambda_{0,V}(z)$ is continuous on $V$ there exists $z_y \in V$ such that
      \begin{equation*}
        \Lambda_{0,V}^*(y) = z_y y - \Lambda_{0,V}(z_y).
      \end{equation*}
      Hence,
      \begin{equation*}
        \abs{\Lambda_{\epsilon,V}(z_y) - (z_y y - \Lambda_{0,V}^*(y))} = \abs{\Lambda_{\epsilon,V}(z) - \Lambda_{0,V}(z)}  \le C_V  \tau(\epsilon)^{\eta(V)},
      \end{equation*}
      and so
      \begin{equation*}\begin{split}
        \Lambda_{0,V}^*(y) &\le C_V  \tau(\epsilon)^{\eta(V)} + z_y y - \Lambda_{\epsilon,V}(z_y)\\
        &\le  C_V  \tau(\epsilon)^{\eta(V)} + \sup_{z \in V} \, (z y - \Lambda_{\epsilon,V}(z)) \\
        &= C_V  \tau(\epsilon)^{\eta(V)} + \Lambda_{\epsilon,V}^*(y).
      \end{split}\end{equation*}
      Applying the same argument but reversing the roles of $\epsilon$ and $0$ yields
      \begin{equation*}
        \abs{\Lambda_{0,V}^*(y) - \Lambda_{\epsilon,V}^*(y)} \le C_V  \tau(\epsilon)^{\eta(V)},
      \end{equation*}
      which concludes the proof.
    \end{proof}
  \end{theorem}

  While Theorem \ref{thm:convergence_of_rate_functions} confirms that $\Lambda_{\epsilon,V}^*$ and $\Lambda_{0,V}^*$ are close, we don't know whether $\Lambda_{0,V}^*$ and $r_g$ are necessarily close. We finish this section with a Proposition that clarifies this relationship, but first we must prove that $z \mapsto \ln(\lambda_0(z))$ is convex on $(-\theta, \theta)$.
  While it is known that $z \mapsto \ln(\lambda_0(z))$ is convex in a small real neighbourhood of $0$ \cite[Proposition VIII.3]{HennionHerve}, we are not aware of a proof for convexity on its entire domain, which in this case is $(-\theta, \theta)$.

  \begin{lemma}\label{lemma:concave_ln_lambda}
    The function $z \mapsto \ln(\lambda_0(z))$ is convex on $(-\theta, \theta)$.
    \begin{proof}
      Adapting arguments from \cite[Proposition 5, Proposition 6]{aimino2015note}, for $z \in (-\theta, \theta)$ we have
      \begin{equation}\label{eq:lemma:concave_ln_lambda_1}
        \ln(\lambda_0(z)) = \lim_{n \to \infty} \frac{1}{n} \ln \intf e^{z S_n(g)} dm.
      \end{equation}
      Let $a,b \in (-\theta, \theta)$ and $t \in [0,1]$. Then \eqref{eq:lemma:concave_ln_lambda_1} yields
      \begin{equation*}
        \ln(\lambda_0(at + (1-t)b)) = \lim_{n \to \infty} \frac{1}{n} \ln \intf e^{S_n(g)(at + (1-t)b)} dm.
      \end{equation*}
      Recalling that $e^{S_n(g)z}$ is positive for $z \in \R$ and applying H{\"o}lder's inequality with conjugate exponents $1/t$ and $1/(1-t)$ yields
      \begin{equation*}\begin{split}
        \frac{1}{n} \ln \int &e^{S_n(g)(at + (1-t)b)} \mathrm{d}m = \frac{1}{n} \ln \intf \left(e^{S_n(g)a}\right)^t \left(e^{S_n(g)b}\right)^{1-t} dm \\
        &\le \frac{1}{n} \ln \left(\intf e^{S_n(g)a} dm\right)^{t}\left(\intf e^{S_n(g)b} dm\right)^{1-t}  \\
        &= \frac{t}{n} \ln \left(\intf e^{S_n(g)a} dm\right) + \frac{1-t}{n}\ln\left(\intf e^{S_n(g)b} dm\right).
      \end{split}\end{equation*}
      Letting $n \to \infty$ yields the required inequality.
    \end{proof}
  \end{lemma}

  \begin{proposition}\label{prop:estimation_of_truncated_rate_func}
    Suppose that we are in the setting of Corollary \ref{thm:application_to_ng_method}. If $V \subseteq (-\theta,\theta)$ is a closed interval then $r_g$ and $\Lambda_{0,V}^*$ are equal on $\Lambda_{0}'(V)$. Consequently, for every compact set $W \subseteq \Lambda_{0}'(-\theta,\theta)$ there exists a closed interval $V_W \subseteq (-\theta,\theta)$ such that $r_g$ and $\Lambda_{0,V_W}^*$ are equal on $W$.
    \begin{proof}
      For brevity let $\Lambda_0(z) = \ln(\lambda_0(z))$. By the definition of the convex conjugate we have
      \begin{equation}\label{eq:estimation_of_truncated_rate_func_1}
        \Lambda_{0,V}^*(y) = \sup_{z \in V} \, ( yz - \Lambda_{0}(z)).
      \end{equation}
      Note that for each $y \in \R$ the function $z \mapsto yz - \Lambda_{0,V}(z)$ that is defined on $(-\theta, \theta)$ is concave by Lemma \ref{lemma:concave_ln_lambda}. By differentiating, we therefore see that if $y \in \Lambda_{0}'(V)$ then the supremum in \eqref{eq:estimation_of_truncated_rate_func_1} is attained by some $z$ satisfying $\Lambda_0'(z) = y$. Hence, for $y \in \Lambda_{0}'(V)$ we have
      \begin{equation*}
          \Lambda_{0,V}^*(y) = y(\Lambda_{0}')^{-1}(y) - \Lambda_{0}(\Lambda_{0}')^{-1}(y)).
      \end{equation*}
      The same argument shows that the same formula holds for $r_g$ on $\Lambda_{0}'(V)$, and so the two functions are equal on $\Lambda_{0}'(V)$, which is a closed interval as $\Lambda_{0}'$ is monotonic. For the second conclusion simply take $V_W$ to be the inverse image under $\Lambda_{0}'$ of the convex hull of $W$, which is clearly a closed interval.
    \end{proof}
  \end{proposition}

  \section{Stability and approximation of statistical laws for piecewise expanding interval maps}
  \label{sec:piecewise_expanding_one_dim}

  In this section we will demonstrate the utility of our theory by applying the results of Sections \ref{sec:abstract_twist_stability} and \ref{sec:application_to_NG} to piecewise expanding interval maps. This setting is classical, so we refer the reader to \cite{boyarsky1997laws, baladi2000positive} for further details. We obtain stability of the variance and rate function under standard classes of perturbations, including perturbations arising from the Ulam numerical scheme \cite{li}.

  Denote by $(\BV,\norm{\cdot})$ the Banach space of functions of bounded variation on $[0,1]$, equipped with the norm $\norm{\cdot} = \Var(\cdot) + \lnorm{\cdot}$.
  We call $T: [0,1] \to [0,1]$ a Lasota-Yorke map if there are $0=a_0<a_1<\cdots< a_r=1$ and $\gamma > 1$ such that for each $i=1,\ldots,r$, $T$ is $C^1$ on $(a_{i-1},a_i)$ with a $C^1$ extension to $[a_{i-1},a_i]$, $\abs{T'_{|(a_{i-1},a_i)}}\ge \gamma$ and $1/\abs{T'_{|(a_{i-1},a_i)}}\in \BV$.
  It is classical that $\BV$ satisfies condition (S) (Definition \ref{def:strong_space_condition}) and that for a topologically mixing Lasota-Yorke map $T$ the associated Perron-Frobenius operator is simple, quasi-compact and has spectral radius 1. Thus $T$ satisfies a CLT and LDP by Theorem \ref{aiminothm} for appropriate observables.
  In the following corollary we confirm the stability of statistical parameters in this setting, as a consequence of Theorems \ref{thm:application_to_ng_method}, \ref{thm:stability_of_variance}, and \ref{thm:convergence_of_rate_functions} and Proposition \ref{prop:estimation_of_truncated_rate_func}.
  \begin{corollary}
  \label{LYthm}
  Let $T$ be a topologically mixing Lasota-Yorke map, $\mu$ its invariant measure, $\{\LL_\epsilon\}_{\epsilon\ge 0}$ a family of operators satisfying (KL) where $\LL_0 = \LL$ is the Perron-Frobenius operator associated with $T$, and $g \in \BV$ be a real-valued observable such that $\intf g d \mu = 0$.
  Then each of the following holds as $\epsilon \to 0$:
  \begin{enumerate}
    \item Theorem \ref{thm:stability_of_variance}: the approximate variance $\lambda_\epsilon^{(2)}(0)$ converges to the true variance $\sigma^2_g$.
    \item Theorem \ref{thm:convergence_of_rate_functions}: if each $\LL_\epsilon$ is positive then for every compact subset $W$ of the domain of $r_g$ there exists a closed interval $V$ such that $\lim_{\epsilon \to 0} \sup_{z \in V} (sz - \ln(\lambda_\epsilon(z))) = r_g(s)$ uniformly for $s \in W$.
  \end{enumerate}
  In particular, the numerical, stochastic, and deterministic perturbations (NP), (SP), and (DP) detailed in the introduction satisfy (KL) and we obtain the corresponding approximation and stability of the variance and rate function under these perturbations.
  \end{corollary}

  \begin{remark}
    Item 2 of Corollary \ref{LYthm} has been proven before for perturbations of type (DP) and (NP) in  \cite{keller2008continuity} and \cite{bahsoun2016rigorous}, respectively.
  \end{remark}

  The map we consider for the numerics in the remainder of this section is the non-Markov piecewise affine map with $a=2.1$.
  \begin{equation}
  \label{doubletenteqn}
  T_a(x)=\left\{
           \begin{array}{ll}
             ax, & \hbox{$0\le x<1/4$;} \\
             -a(x-1/2), & \hbox{$1/4\le x<1/2$;} \\
             -a(x-1/2)+1, & \hbox{$1/2\le x<3/4$;} \\
             a(x-1)+1, & \hbox{$3/4\le x\le 1$.}
           \end{array}
         \right.
  \end{equation}
  By standard arguments, we obtain $\alpha=1/1.05<1$ in the Lasota-Yorke inequality for $\LL$ and thus $\LL$ is quasi-compact.
  Moreover, it is clear from the graph of $T$ (Figure \ref{fig:doubletent}, upper left) that forward images of any interval $I\subset [0,1]$ eventually cover all of $[0,1]$;  thus, the eigenvalue 1 of $\LL$ is simple.

  \captionsetup{width=\linewidth}
  \begin{center}
    \includegraphics[width=.45\linewidth]{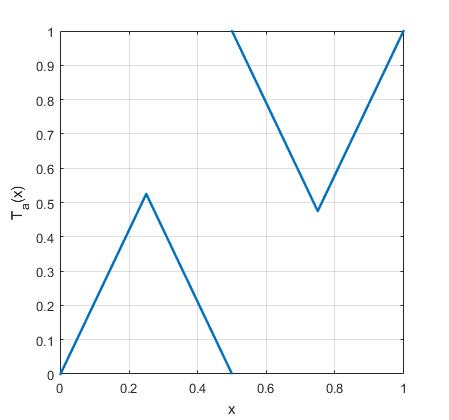}
    \includegraphics[width=.45\linewidth]{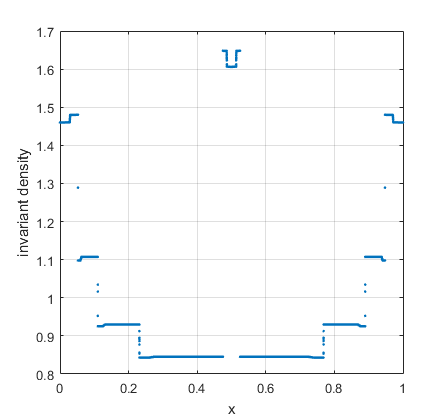}\\
    \includegraphics[width=.45\linewidth]{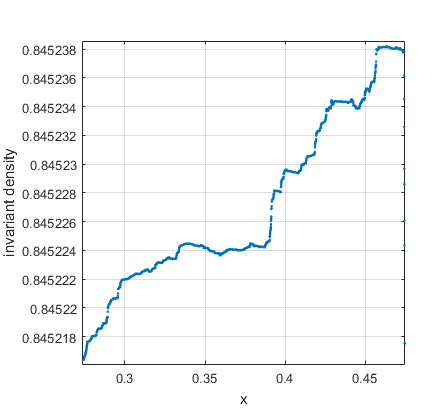}
    \includegraphics[width=.45\linewidth]{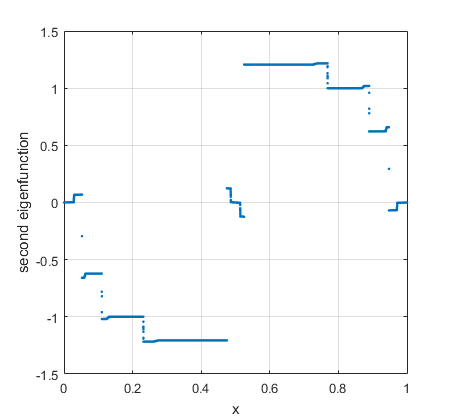}
    \captionof{figure}{Graph of $T_a$ (upper left),  an approximation of the invariant density with $n=25000$ (upper right), a zoom of an apparent ``flat'' section of the invariant density showing fine structure (lower left), and an approximation of the second eigenfunction with eigenvalue $\lambda_2\approx 0.8079$ (lower right).}\label{fig:doubletent}
  \end{center}
  Note that the dynamics of $T_a$ for $a\gtrapprox 2$ has infrequent transitions between the left and right halves of the unit interval;  in such a situation, these sets are sometimes called almost-invariant sets \cite{DJ99,DFS01}.
  We select observables $g$ (taking values approximately in the range $[-1,1]$) that emphasise this structure to varying extents, and illustrate the combined effects of the dynamics and the observable on variances and rate functions;  see Figure \ref{fig:ratefuncs} for graphs of the various $g$.
  For example, the Birkhoff sums of $g(x)=\mathbf{1}_{[0,1/2]}-\mathbf{1}_{(1/2,1]}$ will typically take longer to converge because of frequent long sequences of similar $g$ values (in this case either $1$ or $-1$).
  On the other hand, the observable $g(x)=\cos(2\pi x)-0.0614$ is not strongly correlated with the almost-invariant dynamics and one expects a more rapid convergence of Birkhoff sums.
  These arguments are reflected in the table of variances, Table \ref{tab:variance}, and the graph of rate functions, Figure \ref{fig:ratefuncs} (right).

  \subsection{Numerical schemes}
  \label{sec:1dulampert}

  Our choice of numerical scheme is dictated by the class of map.
  Because we are considering general (non-Markov) Lasota-Yorke maps, the natural choice of Banach space is $\BV$, with the weak and strong norms being the $L^1$ and $\BV$ norms, respectively.
  Since the eigenfunctions of $\LL$ can be discontinuous (see Figure \ref{fig:doubletent} upper right and lower right), we use locally supported functions for our approximation space, and in particular, locally constant functions, leading to the well-known Ulam scheme \cite{ulam}.
  If, on the other hand, we restricted ourselves to globally differentiable, full-branched maps (a smaller and better-behaved class), then it would be natural to work with $C^r$ functions and use a globally supported basis consisting of Chebyshev polynomials (if the phase space is an interval) or trigonometric polynomials / Fourier modes (if the phase space is a circle);  these bases will exploit the smoothness of the map to produce commensurately more accurate estimates.

  For a partition of $[0,1]$ into subintervals $I_1,\ldots,I_n$, setting $\mathcal{B}_n=\mbox{span}\{\mathbf{1}_{I_i}:1\le i\le n\}$, we define the conditional expectation operator $\mathbb{E}_n:L^1([0,1])\to \mathcal{B}_n$ by
  \begin{equation}
  \label{ulamproj}
  \mathbb{E}_nf=\sum_{i=1}^n \frac{\int_{I_i}f\ dm}{m(I_i)}\mathbf{1}_{I_i}.
  \end{equation}
  It is well known (e.g.\ \cite{li}) that the matrix representation of $\mathbb{E}_n\LL$ on $\mathcal{B}_n$ is
  \begin{equation}
  \label{ulammatrix}
  P_{ij}=\frac{m(I_i\cap T^{-1}I_j)}{m(I_j)},
  \end{equation}
  under multiplication on the left.
  In our experiments, we use equipartitions of $[0,1]$ of increasing cardinality $n$.
  Putting $\epsilon=1/n$, we set $\LL_\epsilon = \LL_{1/n}=\mathbb{E}_n\LL$.
  The property \ref{en:kl_conv} is satisfied;  see e.g.\ the discussion in \S16--18 \cite{keller1982stochastic}.
  The operators $\LL_{1/n}$ are Markov for every $n$ and  therefore satisfy \ref{en:kl_l1_bound} and are positive.
  The expectation operator $\mathbb{E}_n$ reduces variation, and thus \ref{en:kl_ly_bound} is also satisfied.
  Our estimate of the twisted operator $\LL(z)$ will be the operator $\LL_{1/n}(z):=\LL_{1/n}M_{g_n}(z)$, where $g_n=\sum_{i=1}^n g((i-1/2)/n)\mathbf{1}_{I_i}$ takes the value of $g$ at the midpoint\footnote{Strictly speaking, Ulam's method for twisted transfer operators will involve integrals of $g$, which can be numerically evaluated. We have chosen the above midpoint approximation of $g$ for computational convenience;  note that the midpoint rule is the same order of accuracy as the trapezoidal method of numerical quadrature and often slightly more accurate (errors are about a factor 1/2 smaller).  We additionally computed the values in Table \ref{tab:variance} with an ``exact'' implementation of Ulam and the errors due to the midpoint estimate of $g$ were several orders of magnitude smaller than the errors due to the overall Ulam discretisation.} of each $I_i$, $i=1,\ldots,n$.

  \subsection{Estimating the variance}
  \label{sec:estimating_variance_1d}

  We numerically evaluate the expression (\ref{eq:variance_for_comp_2}) for $\lambda^{(2)}(0)$.
  The term $(\Id - \LL_{\epsilon})^{-1}\LL_{\epsilon}(gv_{\epsilon,0})$ is numerically determined by solving the single linear system of equations $(\Id - \LL_{\epsilon})v'(0)=\LL_{\epsilon}(gv_{\epsilon,0})$ for the unknown $v'(0)$ (i.e. $\restr{\frac{d}{dz}v}{z=0}$), restricting $v'(0)$ to the co-dimension 1 subspace of zero-Lebesgue-mean functions.
  The MATLAB function for computing the variance is given in Appendix \ref{app:c}.
  \begin{small}
  \begin{table}[hbt]
  \begin{center}
  \begin{tabular}{ccccc}
  \hline
  Ulam &\multicolumn{4}{c}{Variance estimates for observable $g(x)$} \\
  subintervals  & $\cos(2\pi x)-0.0614$ & $2x-1$ & $\sin(2\pi x)$ & $\mathbf{1}_{[0,1/2]}-\mathbf{1}_{(1/2,1]}$ \\
  \hline
  200 &  0.51057 & 4.3355 & 6.5368 & 17.006 \\
  1000&  0.50496 & 4.2886 & 6.4959 & 16.859 \\
  5000&  0.50430 & 4.2871 & 6.4950 & 16.855 \\
  25000& 0.50396 & 4.2860 & 6.4936 & 16.851 \\

  \hline
  \end{tabular}
  \captionof{table}{Computed variances for four different observables and at four different Ulam resolutions.}\label{tab:variance}
  \end{center}
  \end{table}
  \end{small}

  For a transformation $T:[0,1]\to [0,1]$ let \verb"P" denote a row-stochastic Ulam matrix constructed on an equi-partition of $[0,1]$ in MATLAB's \verb"sparse" format.
  To compute an estimate of the variance \verb"v" for the observable $g(x)=\sin(2\pi x)$, use:
  \begin{verbatim}
  obs=@(x)sin(2*pi*x);
  [v,~,~,~,~] = variance(P,obs);
  \end{verbatim}

  There have been a number of prior rigorous numerical estimates of variance for interval maps.
  Bahsoun \emph{et al.} \cite{bahsoun2016rigorous}, Pollicott \emph{et al} \cite{pollicott2017rigorous}, and Wormell \cite{wormell2017spectral} develop algorithms that output an interval in which the variance is guaranteed to lie.
  Bahsoun \emph{et al.} applies to general Lasota-Yorke maps, uses Ulam's method, and employs a ``brute force'' approach of taking high powers of $\LL_{1/n}$ to achieve convergence.
  The method of \cite{pollicott2017rigorous} applies to real analytic expanding (full-branch) maps with real analytic observables, and is based on evaluations on all periodic orbits up to a certain order.
  Wormell \cite{wormell2017spectral} applies to full-branched, $C^3$ expanding maps and uses an approach most similar to ours, with computations in  Chebyshev/Fourier bases.
  In each of these papers, an interval containing the variance of the Lanford map $T(x)=2x+x(1-x)/2\pmod 1$ for the observable $g(x)=x^2$ is obtained.
  The latter two papers, exploiting the analyticity of the map $T$ and observable $g$ can achieve more accurate estimates for the same computational effort.

  In comparison to \cite{bahsoun2016rigorous} we can avoid raising the very sparse matrix $\LL_{1/n}$ to high powers (in the full-branch Lanford map studied in \cite{bahsoun2016rigorous} $\LL_{1/n}^{112}$ is computed).
  We exploit the differentiability properties of the spectral data with respect to the twist parameter (which exist even for general Lasota-Yorke maps) and preserve the high degree of sparseness of $\LL_{1/n}$, which is quickly destroyed by taking powers.
  We only need to solve a single sparse linear equation to obtain an estimate for $\lambda^{(2)}(0)$, which is related to the equation solved in \cite{wormell2017spectral}.
  In comparison to \cite{pollicott2017rigorous} and \cite{wormell2017spectral} we can treat general Lasota-Yorke maps, via the flexible choice of a locally supported basis, however, as explained above, for smoother classes of maps as in \cite{pollicott2017rigorous,wormell2017spectral}, one should adapt the basis accordingly as the Ulam basis will not be competitive with specialised approaches.
  Our variance estimates rigorously converge to the true value as $n\to \infty$; and while it is likely possible to provide an ``interval of guarantee'', as in the above methods, we have not pursued this here.

  \subsection{Estimating the rate function}
  \label{sec:est_rate_func}

  For a fixed value of $s$, we estimate $r_g(s)=-\min_z (\log \lambda(z)-zs)$ by applying MATLAB's built-in unconstrained function minimising routine \verb"fminunc" to the function $f(z)=\log\lambda(z)-sz$.
  We use the default quasi-newton algorithm option for \verb"fminunc" (we found the trust-region algorithm used slightly more iterates) and supply an expression for the first derivative of $f(z)$ with respect to $z$, namely $\phi(z)(g\lambda(z)v(z))-s$ (here $\phi(z)$ and $v(z)$ are the leading left and right, respectively, eigenvectors of $\LL(z)$);  all other settings are the defaults.
  Each evaluation of $f(z)$ requires the computation of $\lambda(z)$ (we obtain $v(z)$ at the same time) and each evaluation of $f'(z)$ requires an additional computation of $\phi(z)$.
  These two eigencomputations are made by simply repeatedly iterating $v(0)$ and $\phi(0)$ with $\LL(z)$ and $\LL(z)^*$ (and normalising), respectively until the change in the estimated eigenvalue is below a tolerance (we used $5\times 10^{-12}$).
  This is relatively efficient because the Ulam matrix approximation of $\LL(z)$ is very sparse, and we found this is also faster than using MATLAB's built-in \verb"eigs" routine to find the single leading eigenvalue.
  We estimate $r_g(s)$ on a grid of $s$ values (in our experiments $s$ ranges from 0 to 0.8 in steps of 0.01), stepping from one grid point to the next.
  We use the previous optimal $z$ as the initial seed for the quasi-newton algorithm to find the optimal $z$ for the next $s$ grid point, and found this choice results in slightly fewer quasi-newton steps than choosing a fixed initialisation.

  For a transformation $T:[0,1]\to [0,1]$ let \verb"P" denote a row-stochastic Ulam matrix constructed on an equipartition of $[0,1]$ in MATLAB's \verb"sparse" format.
  To compute estimates of the rate function $r_g(s)$, for the observable $g(x)=\sin(2\pi x)$, at $s\in[0,0.8]$ spaced 0.01 apart, and store these estimates in a vector \verb"r", use:
  \begin{verbatim}
  s=0:.01:.8;
  obs=@(x)sin(2*pi*x);
  [r,~] = rate_function(s,P,obs);
  \end{verbatim}
  The necessary MATLAB functions are given in Appendix \ref{app:c};  to run the above code to compute these 81 values of the rate function takes\footnote{On a 7th-generation intel core i5 processor.} approximately 1, 4, and 12 seconds for Ulam matrices of sizes 1000, 5000, and 25000, respectively.
  We use the same set of four observables $g$ as in the variance computations (Figure \ref{fig:ratefuncs} left).
  The corresponding rate functions are shown in Figure \ref{fig:ratefuncs} right).
  \begin{center}
    \includegraphics[height=0.4\linewidth]{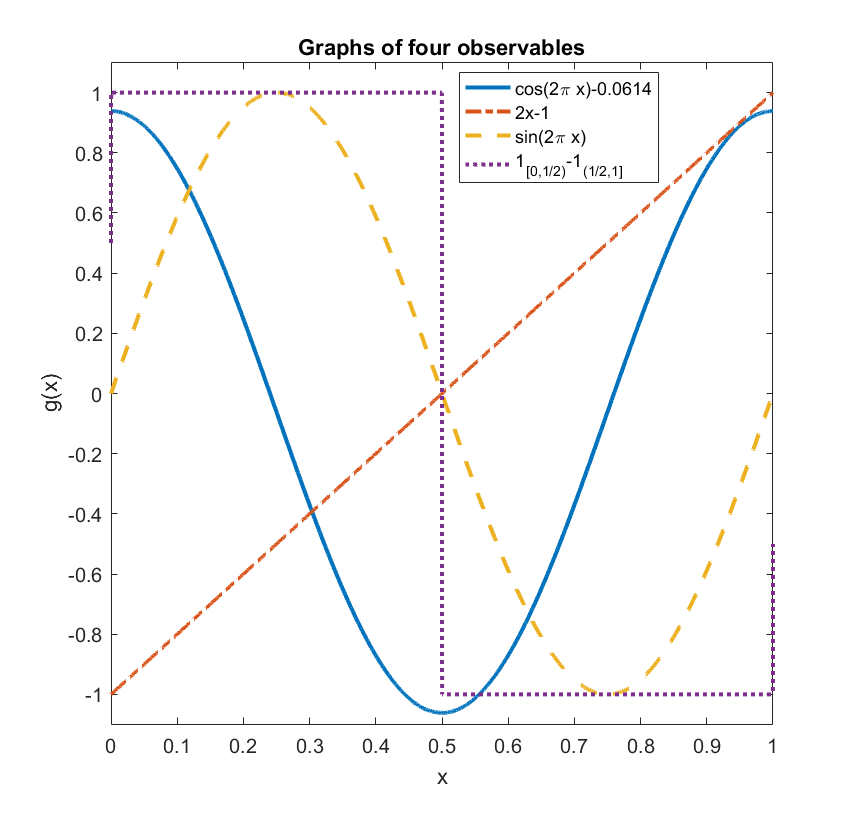}
    \includegraphics[height=0.4\linewidth]{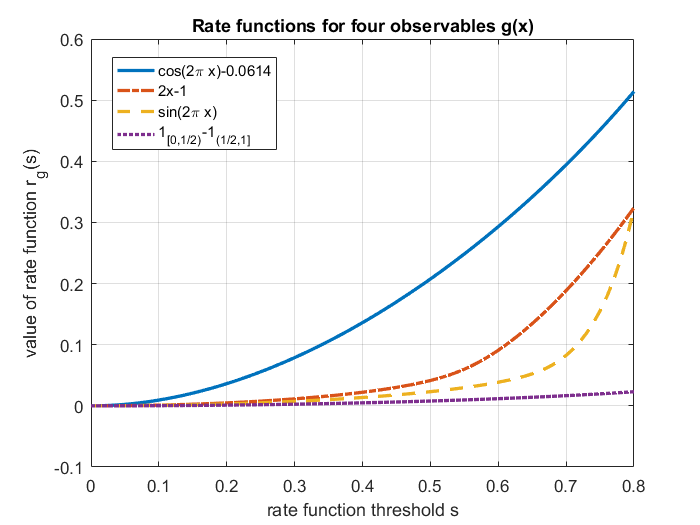}
    \captionof{figure}{Graphs of four different observables (left) and corresponding rate functions (right), computed with $n=1000$.}\label{fig:ratefuncs}
  \end{center}
  Note that the four observables $g$ yield rate functions of increasingly lower value (higher likelihood of large deviations occurring).
  This corresponds to the correlation between the value of the observable and the almost-invariant sets $[0,1/2],[1/2,1]$.
  The observable $g(x)=\cos(2\pi x)-0.0614$ is not particularly correlated with the almost-invariant sets and thus large deviations have low probability.
  On the other hand, the observables $2x-1$ and $\sin(2\pi x)$ have moderate correlation with the almost-invariant sets and large deviations have an increased probability of occurring (interestingly, there is a crossover of these two rate functions around $s=0.8$).
  The observable $g(x)=\mathbf{1}_{[0,1/2]}-\mathbf{1}_{(1/2,1]}$ is very strongly correlated with the almost-invariant sets and we see a correspondingly small rate function.
  Figure \ref{fig:errors} shows the decrease in errors relative to $n=25000$ for the calculations using $n=200, 1000$, and $5000$, typically with somewhat larger errors for larger thresholds $s$, as expected.
  \begin{center}
    \includegraphics[width=\textwidth]{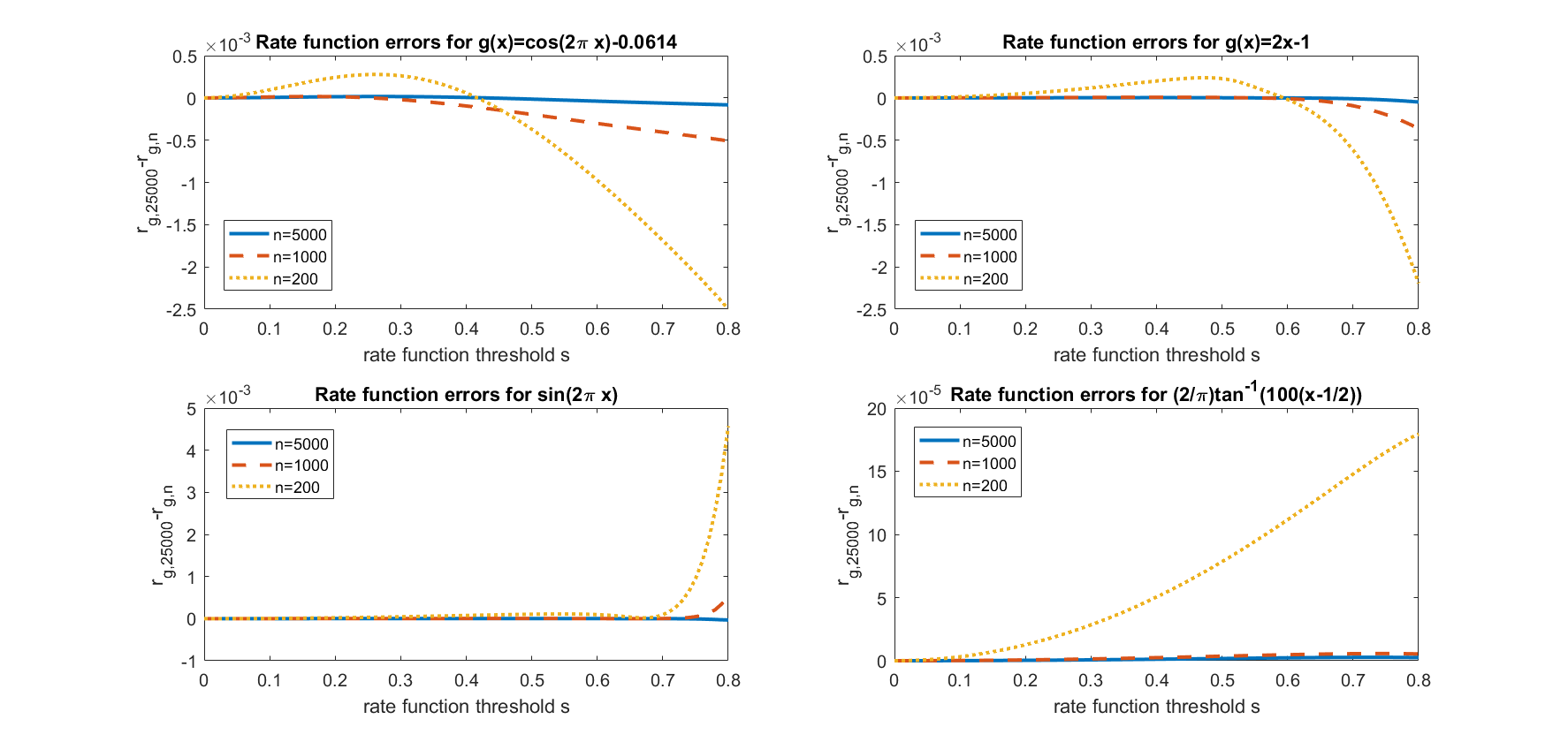}
    \captionof{figure}{Differences between rate function estimates for $n=200, 1000, 5000$ and the rate function estimate using $n=25000$.}\label{fig:errors}
  \end{center}
  We are not aware of prior rigorous numerical methods for estimating rate functions for deterministic dynamics. Prior work on estimating rate functions includes \cite{johnson_meneveau,rohwer_etal}, which use the Legendre transform but not the spectral approach we use here.

  Finally, we note that the rate of escape from the interval $[0,1/2]$ can be estimated via the observable $g(x)=\mathbf{1}_{[0,1/2]}-\mathbf{1}_{(1/2,1]}$ by computing the rate function for a threshold very close to $1$.
  With $n=1000$, and an $s$ threshold of $1-10^{-15}$, one obtains a rate function value of $0.04879016416945$ (in this experiment, the optimality tolerance in \verb"rate_function.m" was decreased to $10^{-13}$).
  Alternatively, computing the negative logarithm of the leading eigenvalue of the transfer operator restricted to the interval $[0,1/2]$ (this is particularly straightforward with an Ulam basis, see e.g.\ \cite{bahsounescape,ulamescape}), one obtains $0.04879016416943$.
  Thus, the rate function calculation and the escape rate calculation are consistent up to 13 decimal places for an Ulam matrix of size $n=1000$ (note we are not claiming accuracy of the true values up to this precision).

  \section{Stability of statistical laws for piecewise expanding maps in multiple dimensions}
  \label{sec:piecewise_expanding_multi_dim}

  In this section we consider the class of multi-dimensional piecewise expanding maps studied in \cite{saussol2000absolutely} by Saussol. Using the fact that  the space of quasi-H{\"o}lder functions satisfies condition (S) -- in particular, it is a Banach algebra -- it is shown in \cite{aimino2015note} that such maps obey a CLT and LDP for quasi-H{\"o}lder observables via the Nagaev-Guivarc'h method. Using the theory developed in Section \ref{sec:application_to_NG}, we prove that the invariant density, variance and rate function of these maps are stable with respect to perturbations satisfying (KL). We then give two examples of classes of perturbations satisfying (KL) in this setting: stochastic perturbations and perturbations arising via Ulam's method.

  Let $d \ge 2$ and denote by $m$ the Lebesgue measure on $\R^d$. Let $\Omega \subseteq \R^d$ be compact and satisfy $\cl{\intr(\Omega)} = \Omega$. Without loss of generality we may assume that $m(\Omega) = 1$. Let $T: \Omega \to \Omega$ be a piecewise expanding map in the sense of \cite{saussol2000absolutely}. To prevent clashes with our notation, we denote the constants $\alpha$ and $\epsilon_0$ associated with $T$ in \cite{saussol2000absolutely} by $\beta$ and $\eta_0$, respectively. We note that $\beta \in (0, 1]$ and $\eta_0 > 0$.

  We now define the space of quasi-H{\"o}lder functions, reproducing some of \cite[Section 4]{aimino2015note} for the reader's convenience. Let $A \subseteq \R^d$ be an arbitrary Borel set. For each $f \in L^1(\R^d)$ the oscillation of $f$ over $A$ is defined to be
  \begin{equation*}
    \osc{f, A} = \esup_{(y_1,y_2) \in A\times A} \abs{f(y_1) - f(y_2)},
  \end{equation*}
  where the essential supremum is taken with respect to the product measure $m \times m$ on $A \times A$. For every $f \in L^1(\R^d)$ and $\eta > 0$ the function $x \mapsto \osc{f, B_\eta(x)}$ is well defined and lower-semicontinuous, and therefore also measurable. Let
  \begin{equation*}\begin{split}
    &\abs{f}_{\beta} = \sup_{0 < \eta \le \eta_0} \eta^{-\beta} \intf_{\R^d} \osc{f, B_\eta(x)} dx \quad \text{and,}\\ &\Va(\R^d) = \{ f \in L^1(\R^d) : \abs{f}_\beta < \infty \}.
  \end{split}\end{equation*}
  The space of quasi-H{\"o}lder functions on $\Omega$ is defined to be
  \begin{equation*}
    \Va(\Omega) = \{ f \in \Va(\R^d) : \supp(f) \subseteq \Omega \},
  \end{equation*}
  and is a Banach space when endowed with the norm $\norm{\cdot}_{\beta} = \lnorm{\cdot} + \abs{\cdot}_\beta$ \cite{saussol2000absolutely,keller1985generalized}.

  \begin{proposition}[{\cite[Section 4]{aimino2015note}}, {\cite[Theorem 5.1]{saussol2000absolutely}}]\label{prop:quasiholder_condition_S}
    The space $\Va(\Omega)$ satisfies condition (S) and the Perron-Frobenius operator $\LL$ associated with the map $T$ is a quasi-compact operator on $L(\Va(\Omega))$ with $\rho(\LL) = 1$.
  \end{proposition}

  Hence, if $\LL$ a simple quasi-compact operator then $T$ obeys both a CLT and LDP by Theorem \ref{aiminothm}. Consequently, due to the theory developed in Section \ref{sec:application_to_NG} we obtain the stability of the statistical parameters associated with these laws under perturbations satisfying (KL).

  \begin{theorem}\label{thm:stability_statistical_props_quasiholder}
    If the Perron-Frobenius operator $\LL$ is simple and quasi-compact, and $\{\LL_\epsilon\}_{\epsilon \ge 0}$ satisfies (KL), with $\LL_0 = \LL$, then each of the following holds as $\epsilon \to 0$:
    \begin{enumerate}
      \item Proposition \ref{prop:keller_liverani_summary}: the leading eigenvector $v_\epsilon$ of $\LL_\epsilon$ converges to the unique $T$-invariant density $\frac{\mathrm{d}\mu}{\mathrm{d}m}$ in $\lnorm{\cdot}$.
      \item Theorem \ref{thm:stability_of_variance}: for each real-valued $g \in \Va(\Omega)$ such that $\intf g d \mu = 0$ the approximate variance $\lambda_\epsilon^{(2)}(0)$ converges to the true variance $\sigma^2_g$.
      \item Theorem \ref{thm:convergence_of_rate_functions}: if each $\LL_\epsilon$ is positive then for every real-valued $g \in \Va(\Omega)$ such that $\intf g d \mu = 0$ and compact subset $W$ of the domain of $r_g$ there exists a closed interval $V$ such that $\lim_{\epsilon \to 0} \sup_{z \in V} \,(sz - \ln(\lambda_\epsilon(z))) = r_g(s)$ uniformly for $s \in W$.
    \end{enumerate}
    \begin{proof}
      As noted, the space $\Va(\Omega)$ and Perron-Frobenius $\LL$ satisfy the conditions of Theorem \ref{aiminothm}. We note that $\Va(\Omega) \subseteq L^\infty(m)$ by \cite[Proposition 3.4]{saussol2000absolutely} and so every $g \in\Va(\Omega)$ is also in $\Va(\Omega) \cap L^\infty(m)$. Hence, Corollary \ref{thm:application_to_ng_method} holds, which immediately yields the stability of the variance by Theorem \ref{thm:stability_of_variance}. In the case where each $\LL_\epsilon$ is a positive operator the stability of the rate function follows by Theorem \ref{thm:convergence_of_rate_functions} and Proposition \ref{prop:estimation_of_truncated_rate_func}. Proposition \ref{prop:keller_liverani_summary} is an immediate consequence of (KL) holding for $\{\LL_\epsilon\}_{\epsilon \ge 0}$.
    \end{proof}
  \end{theorem}

  We will now gather some important properties of $\Va(\Omega)$ for later use. The first result concerns the continuity of inclusion of $\Va(\Omega)$ into $L^\infty(m)$.

  \begin{proposition}[{\cite[Proposition 3.4]{saussol2000absolutely}}]
    \label{prop:injection_l_infty}
    Let $\nu_d$ denote the measure of the $d$-dimensional unit ball i.e. $\nu_d = m(B_1(0))$. For every $f \in \Va(\Omega)$ we have
    \begin{equation*}
      \norm{f}_{L^\infty} \le \frac{1}{\nu_d \eta_0^d}(\eta_0^\beta\abs{f}_\beta + \lnorm{f}) \le \frac{\max\{1, \eta_0^\beta\}}{\nu_d\eta_0^d} \norm{f}_\beta.
    \end{equation*}
  \end{proposition}

  In later sections we will prove that \ref{en:kl_ly_bound} holds for certain perturbations of the Perron-Frobenius operator by making use of the analogous inequality for the unperturbed operator.

  \begin{proposition}[{\cite[Lemma 4.1]{saussol2000absolutely}}]\label{prop:ly_unperturbed_quasiholder}
    Provided that $\eta_0$ is small enough, there exists $\gamma < 1$ and $D < \infty$ such that for each $f \in \Va(\Omega)$ we have
    \begin{equation*}
      \abs{\LL f}_\beta \le \gamma \abs{f}_\beta + D \lnorm{f}.
    \end{equation*}
  \end{proposition}

  \begin{remark}
    We assume that $\eta_0$ is sufficiently small so that the conclusion of Proposition \ref{prop:ly_unperturbed_quasiholder} holds.
  \end{remark}

  \begin{remark}
    In \cite{saussol2000absolutely} the space $\Va(\Omega)$ consists of real-valued functions only and so the proof of \cite[Lemma 4.1]{saussol2000absolutely} only applies to real-valued $f \in \Va(\Omega)$. Examining the proof of \cite[Lemma 4.1]{saussol2000absolutely}, we note that the same conclusion holds for complex-valued $f$ after minor modifications to the arguments. In particular, the essential infimum in \cite[Proposition 3.2 (iii)]{saussol2000absolutely} must be replaced by an essential supremum and consequently the resulting essential supremum term that appears when bounding $R_i^{(1)}(x)$ must be bounded by $\abs{f(y_i)} + \osc{f, B_{s\epsilon}(y_i)}$. The rest of the argument holds \emph{mutatis mutandis}.
  \end{remark}

  \subsection{Ulam perturbations}

  We will now prove that Ulam approximations of the Perron-Frobenius operator satisfy (KL) on $\Va(\Omega)$. As a consequence, the theory developed in Section \ref{sec:application_to_NG} yields the stability of the rate function and variance with respect to numerical approximation by Ulam's method. While \ref{en:kl_l1_bound} is standard, \ref{en:kl_conv} and \ref{en:kl_ly_bound} require significantly more work. In particular, the proof of \ref{en:kl_ly_bound} is quite long and depends critically on the geometry of the partitions inducing the Ulam approximations, so we defer its proof to Appendix \ref{sec:ulam_quasiholder_boundedness}. The class of partitions we consider is the following.

  \begin{definition}\label{def:regular_admissible_partitions}
    For $\kappa \ge 1$ let $\mathcal{P}(\kappa)$ be the collection of finite measurable partitions $Q$ of $\Omega$ satisfying the following conditions:
      \begin{enumerate}
        \item The elements of $Q$ are compact and convex polytopes with non-empty interiors.
        \item For every $I \in Q$ we have $\diam(Q) \le \kappa\diam(B_I)$, where $B_I$ is a ball of maximal volume inscribed in $I$ and $\diam(Q) = \max_{J \in Q} \diam(J)$.
      \end{enumerate}
  \end{definition}

  Each $Q \in \mathcal{P}(\kappa)$ induces a conditional expectation operator $\E_Q$ that is given by
  \begin{equation*}
    \E_Q f = \sum_{I \in Q} \avg{f}{I} \bfone_I,
  \end{equation*}
  where $\hat{f}_K$ denotes the expected value of $f$ on some measurable set $K$ i.e.
  \begin{equation*}
    \avg{f}{K} = \frac{1}{m(K)} \intf_K f dm.
  \end{equation*}
  We adopt the convention that $\avg{f}{K} = 0$ if $m(K) = \infty$. As in Section \ref{sec:1dulampert}, $\LL_{Q} = \E_Q \LL$ is the Ulam approximation of $\LL$ induced by $Q$. If $\{Q_\epsilon\}_{\epsilon  > 0} \subseteq \mathcal{P}(\kappa)$ is a sequence of partitions then we define the corresponding sequence of perturbed Perron-Frobenius operators by $\LL_\epsilon = \E_{Q_\epsilon}\LL$, with $\LL_0 = \LL$ as usual.

  \begin{theorem}[Stability of statistical parameters under Ulam approximations]\label{thm:stability_quasiholder_ulam}
    Suppose that $\LL$ is a simple quasi-compact operator. Let $\{Q_\epsilon\}_{\epsilon  > 0} \subseteq \mathcal{P}(\kappa)$ be such that $\lim_{\epsilon \to 0} \diam(Q_\epsilon) = 0$, and let $\{\LL_\epsilon\}_{\epsilon \ge 0}$ be the corresponding Ulam approximations of the Perron-Frobenius operator. If
    \begin{equation*}
      2d \left(1 + \frac{\kappa}{\sqrt[d]{\frac{3}{2}} - 1} \right)^\beta < \frac{1}{\gamma},
    \end{equation*}
    where the constant $\gamma$ is from Proposition \ref{prop:ly_unperturbed_quasiholder},
    then there exists $\epsilon_0 > 0$ such that $\{\LL_\epsilon\}_{0 \le \epsilon \le \epsilon_0}$ satisfies (KL). Consequently, the conclusion of Theorem \ref{thm:stability_statistical_props_quasiholder} holds.
  \end{theorem}

  \begin{example}
    Let $\Omega = [0,1]^d$. For each $n \in \Z^+$ the set
    \begin{equation*}
      Q_n = \left\{ \frac{\Omega}{n} + b : b \in \frac{\Z^{d}}{n} \right\}
    \end{equation*}
    is a measurable partition of $\Omega$ consisting of cubes congruent to $[0,1/n]^d$.
    It is straightforward to show that $Q_n \in \mathcal{P}\left(\sqrt{d}\right)$ for every $n \in \Z^+$.
    Thus, Theorem \ref{thm:stability_quasiholder_ulam} applies to any piecewise expanding map (in the sense of \cite{saussol2000absolutely}) $T: \Omega \to \Omega$ such that
    \begin{equation*}
      2d \left(1 + \frac{\sqrt{d}}{\sqrt[d]{\frac{3}{2}} - 1} \right)^\beta < \frac{1}{\gamma},
    \end{equation*}
    where $\gamma$ is from Proposition \ref{prop:ly_unperturbed_quasiholder} and $\beta$ is equal to the constant $\alpha$ appearing in \cite[Section 2, (PE2)]{saussol2000absolutely}.
  \end{example}

  \begin{proposition}[\ref{en:kl_conv} for Ulam approximations]\label{prop:kl_conv_quasiholder_ulam}
    If $\{Q_\epsilon\}_{\epsilon  > 0} \subseteq \mathcal{P}(\kappa)$ is a sequence of partitions with $\lim_{\epsilon \to 0} \diam(Q_\epsilon) = 0$, then there exists $\epsilon_1 > 0$ such that for every $\epsilon \in [0, \epsilon_1]$ we have
    \begin{equation*}
      \tnorm{\LL_\epsilon - \LL_0} \le 2\diam(Q_\epsilon)^\beta \norm{\LL}_\beta.
    \end{equation*}
    In particular $\{\LL_\epsilon\}_{0 \le \epsilon \le \epsilon_1}$ satisfies \ref{en:kl_conv}.
    \begin{proof}
      The statement is clearly true for $\epsilon = 0$, so we may assume that $\epsilon > 0$. One has
      \begin{equation}\label{eq:kl_conv_quasiholder_ulam_1}
        \tnorm{\E_{Q_\epsilon} - \Id} = \sup_{\norm{f}_\alpha = 1} \sum_{I \in Q_\epsilon} \intf_I \abs{\avg{f}{I} - f} dm.
      \end{equation}
      Fix $I \in Q_\epsilon$. Let $f_r$ and $f_i$ be the real and imaginary parts of $f \in \Va(\Omega)$, respectively. By linearity of integration and the triangle inequality we have
      \begin{equation*}
        \intf_I \abs{\avg{f}{I} - f} dm \le \intf_I \abs{\avg{(f_r)}{I} - f_r} dm + \intf_I \abs{\avg{(f_i)}{I} - f_i} dm.
      \end{equation*}
      Applying this to \eqref{eq:kl_conv_quasiholder_ulam_1} yields
      \begin{equation}\label{eq:kl_conv_quasiholder_ulam_2}
        \tnorm{\E_{Q_\epsilon} - \Id} \le 2 \sup \left\{ \sum_{I \in Q_\epsilon} \intf_I \abs{\avg{f}{I} - f} dm : \substack{f \text{ is real valued}\\ \text{and }\norm{f}_\alpha = 1}\right\}.
      \end{equation}
      Let $f \in \Va(\Omega)$ be real valued and fix $x \in I$. Then for almost every $y_1, y_2 \in I$ we have
      \begin{equation*}
        \abs{f(y_1) - f(y_2)} \le \osc{f, B_{\diam(Q_\epsilon)}(x)}.
      \end{equation*}
      Taking the expectation with respect to $y_1$ over $I$ we find that
      \begin{equation*}
        \abs{\avg{f}{I} - f(y_2)} \le \frac{1}{m(I)}\intf_I \abs{f(y_1) - f(y_2)} dy_1 \le \osc{f, B_{\diam(Q_\epsilon)}(x)},
      \end{equation*}
      for almost every $y_2 \in I$. Taking the expectation with respect to both $y_2$ and $x$ over $I$ we then obtain
      \begin{equation*}
        \intf_I \abs{\avg{f}{I} - f} dm \le \intf_I \osc{f, B_{\diam(Q_\epsilon)}(x)} dx.
      \end{equation*}
      As $\lim_{\epsilon \to 0} \diam(Q_\epsilon) = 0$ there exists $\epsilon_1>0$ such that $\diam(Q_\epsilon) \le \eta_0$ for all $\epsilon \in (0, \epsilon_1]$. Recalling the definition of $\norm{\cdot}_\beta$, for each $\epsilon \in (0, \epsilon_1]$ we have
      \begin{equation*}
        \sum_{I \in Q_\epsilon} \intf_I \abs{\avg{f}{I} - f} dm \le \intf_{\R^d} \osc{f, B_{\diam(Q_\epsilon)}(x)} dx \le \diam(Q_\epsilon)^\beta \norm{f}_\beta,
      \end{equation*}
      which, when applied to \eqref{eq:kl_conv_quasiholder_ulam_2}, then yields
      \begin{equation*}
        \tnorm{\E_{Q_\epsilon} - \Id} \le 2\diam(Q_\epsilon)^\beta.
      \end{equation*}
      We conclude the proof by noting that $\tnorm{\LL_\epsilon - \LL} \le \tnorm{\E_{Q_\epsilon} - \Id} \norm{\LL}_\beta$.
    \end{proof}
  \end{proposition}

  \begin{proposition}[\ref{en:kl_l1_bound} for Ulam approximations]\label{prop:l1_bound_quashiolder_ulam}
    If $\{Q_\epsilon\}_{\epsilon  > 0} \subseteq \mathcal{P}(\kappa)$ is a sequence of partitions, then for each $\epsilon \ge 0$ and $n \in Z^+$ we have
    \begin{equation*}
      \lnorm{\LL_\epsilon^n} \le 1.
    \end{equation*}
    In particular, $\{\LL_\epsilon\}_{\epsilon \ge 0}$ satisfies \ref{en:kl_l1_bound}.
    \begin{proof}
      Note that $\lnorm{\LL} = 1$ and that, for each $\epsilon > 0$, we have $\lnorm{\E_{Q_\epsilon}} =  1$. The conclusion follows immediately.
    \end{proof}
  \end{proposition}

  We will now verify \ref{en:kl_ly_bound} for the Ulam approximations. The main technical requirements are the following two lemmas, which we prove in Appendix \ref{sec:ulam_quasiholder_boundedness}.

  \begin{lemma}\label{cor:conditional_exp_quasi_holder_uniform_bound}
    If $\{Q_\epsilon\}_{\epsilon  > 0} \subseteq \mathcal{P}(\kappa)$ is a sequence of partitions with $\lim_{\epsilon \to 0} \diam(Q_\epsilon) = 0$ then there exists $\epsilon_2 > 0$ such that
    \begin{equation*}
      \sup_{0 < \epsilon \le \epsilon_2} \abs{\E_{Q_\epsilon}}_\beta \le 2d \left(1 + \frac{\kappa}{\sqrt[d]{\frac{3}{2}} - 1} \right)^\beta.
    \end{equation*}
  \end{lemma}

  \begin{lemma}\label{lemma:bounded_cond_exp_quasiholder}
    If $Q \in \mathcal{P}(\kappa)$ then $\E_Q \in L(\Va(\Omega))$.
  \end{lemma}

  \begin{proposition}[\ref{en:kl_ly_bound} for Ulam approximations]\label{prop:kl_ly_quasiholder_ulam}
    Under the hypotheses of Theorem \ref{thm:stability_quasiholder_ulam} there exists $\alpha \in (0,1)$ and $C_3 > 0$ such that for all $f \in \Va(\Omega)$, $n \in \Z^+$ and $\epsilon \in [0,\epsilon_2]$ we have
    \begin{equation*}
      \norm{\LL_\epsilon^n f}_\beta \le \alpha^n \norm{f}_{\beta} + C_3 \lnorm{f}.
    \end{equation*}
    \begin{proof}
      By Proposition \ref{prop:ly_unperturbed_quasiholder} we have
      \begin{equation*}
        \abs{\LL f}_\beta \le \gamma \abs{f}_\beta + D \lnorm{f},
      \end{equation*}
      where $\gamma < 1$ and $D < \infty$. Let $\epsilon_2$ be as in Lemma \ref{cor:conditional_exp_quasi_holder_uniform_bound}. If $\epsilon \in [0, \epsilon_2]$ then
      \begin{equation}\label{eq:kl_ly_quasiholder_ulam_1}
        \abs{\LL_\epsilon f}_\beta \le \left(\sup_{0 < \epsilon \le \epsilon_2} \abs{\E_{Q_\epsilon}}_\beta \right)\abs{\LL f}_\beta
        \le \alpha \abs{f}_\beta + C \lnorm{f},
      \end{equation}
      where
      \begin{equation*}
        \alpha = \gamma 2d \left(1 + \frac{\kappa}{\sqrt[d]{\frac{3}{2}} - 1}\right)^\beta \text{ and } C = 2dD \left(1 + \frac{\kappa}{\sqrt[d]{\frac{3}{2}} - 1}\right)^\beta.
      \end{equation*}
      Note that $\alpha < 1$ by the hypotheses of Theorem \ref{thm:stability_quasiholder_ulam}. By iterating \eqref{eq:kl_ly_quasiholder_ulam_1} and applying Proposition \ref{prop:l1_bound_quashiolder_ulam}, for each $n \in \Z^+$ we have
      \begin{equation*}
        \abs{\LL_\epsilon^n f}_\beta \le \alpha^n \abs{f}_\beta + \frac{C}{1-\alpha} \lnorm{f}.
      \end{equation*}
      Finally, recalling the definition of $\norm{\cdot}_\beta$ and applying Proposition \ref{prop:l1_bound_quashiolder_ulam} again yields
      \begin{equation*}
        \norm{\LL_\epsilon^n f}_\beta \le \alpha^n \norm{f}_\beta + \left(1 + \frac{C}{1-\alpha}\right) \lnorm{f}.
      \end{equation*}
    \end{proof}
  \end{proposition}

  \begin{proof}[Proof of Theorem \ref{thm:stability_quasiholder_ulam}]
     With $\epsilon_1$ as in Proposition \ref{prop:kl_conv_quasiholder_ulam} and $\epsilon_2$ as in Proposition \ref{prop:kl_ly_quasiholder_ulam}, set $\epsilon_0 = \min\{\epsilon_1, \epsilon_2\}$. By Lemma \ref{lemma:bounded_cond_exp_quasiholder} and Propositions \ref{prop:kl_conv_quasiholder_ulam}, \ref{prop:l1_bound_quashiolder_ulam} and \ref{prop:kl_ly_quasiholder_ulam} the family of operators $\{ \LL_\epsilon \}_{0 \le \epsilon \le \epsilon_0}$ satisfies (KL). Upon noting that each $\LL_\epsilon$ is positive the required result follows from Theorem \ref{thm:stability_statistical_props_quasiholder}.
  \end{proof}

  \subsection{Stochastic perturbations}
  \label{sec:1dstochpert}

  We now consider the case where the Perron-Frobenius operator is perturbed by convolution with a stochastic kernel.

  \begin{definition}
    We say that $\{ q_\epsilon \}_{\epsilon > 0} \subseteq L^1(\R^d)$ approximates the identity if
    \begin{enumerate}
      \item Each $q_\epsilon$ is non-negative and satisfies $\lnorm{q_\epsilon} = 1$.
      \item For every $\delta > 0$ we have
      \begin{equation*}
        \lim_{\epsilon \to 0} \intf_{\abs{x} \ge \delta} q_\epsilon(x) dx = 0.
      \end{equation*}
    \end{enumerate}
  \end{definition}

  Let $\{ q_\epsilon \}_{\epsilon > 0}$ approximate the identity. We define the corresponding stochastically perturbed Perron-Frobenius operators by $\LL_\epsilon = (q_\epsilon * \LL) \bfone_{\Omega}$. As usual we set $\LL_0 = \LL$. Our main result for this section is the following.

  \begin{theorem}[Stability of statistical parameters under stochastic perturbations]\label{thm:stability_quasiholder_stochastic}
    Suppose that $\LL$ is a simple quasi-compact operator. Let $\{ q_\epsilon \}_{\epsilon > 0}$ approximate the identity, and let $\{\LL_\epsilon\}_{\epsilon \ge 0}$ be the corresponding stochastically perturbed Perron-Frobenius operators. If
    \begin{equation*}
      \left(1 + \frac{1}{\nu_d \eta_0^{d-\beta}} \sup_{0 < \eta \le \eta_0} \eta^{-\beta} m(B_\eta(\partial\Omega)) \right) < \frac{1}{\gamma},
    \end{equation*}
    where the constant $\gamma$ is from Proposition \ref{prop:ly_unperturbed_quasiholder},
    then $\{\LL_\epsilon\}_{\epsilon \ge 0}$ satisfies (KL). Consequently, the conclusion of Theorem \ref{thm:stability_statistical_props_quasiholder} holds.
  \end{theorem}

  \begin{remark}
    It is not obvious when $\sup_{0 < \eta \le \eta_0 } \eta^{-\beta}m(B_\eta(\partial \Omega)) < \infty$. This is the case if, for example, $\Omega$ is convex \cite[Theorem 6.6]{gruber2010convex}.
  \end{remark}

  As in the previous section, we prove Theorem \ref{thm:stability_quasiholder_stochastic} by showing that $\{\LL_\epsilon \}_{\epsilon \ge 0}$ satisfies (KL).

  \begin{proposition}[\ref{en:kl_conv} for stochastic perturbations]\label{prop:kl1_stochastic_quasiholder}
    If $\{q_\epsilon\}_{\epsilon > 0}$ approximates the identity and $\{\LL_\epsilon\}_{\epsilon \ge 0}$ denotes the corresponding stochastically perturbed Perron-Frobenius operators, then
    \begin{equation*}
      \lim_{\epsilon \to 0} \tnorm{\LL_\epsilon - \LL_0} = 0.
    \end{equation*}
    \begin{proof}
      We have
      \begin{equation} \label{eq:kl1_stochastic_quasiholder_1}
        \tnorm{\LL_\epsilon - \LL_0} \le \sup_{\norm{f}_\beta=1} \intf_\Omega \intf_{\R^d} \abs{(\LL f)(x-y) - (\LL f)(x)}q_{\epsilon}(y) dy  dx.
      \end{equation}
      Let $\delta \in (0, \eta_0)$. We break the inner integral in \eqref{eq:kl1_stochastic_quasiholder_1} into two parts: the component where $y \in B_\delta(0)$ and the component where $y \in \R^d \setminus B_\delta(0)$. Dealing with the first component we have
      \begin{equation}\begin{split} \label{eq:kl1_stochastic_quasiholder_2}
        \int_\Omega \int_{B_\delta(0)} &\abs{(\LL f)(x-y) - (\LL f)(x)}q_{\epsilon}(y)\, \mathrm{d}y \,\mathrm{d}x \\
        &\le \intf_\Omega \osc{\LL f, B_\delta(x)} dx \le \delta^\beta \norm{\LL f}_\beta.
      \end{split}\end{equation}
      Whereas for the second component we find
      \begin{equation}\begin{split} \label{eq:kl1_stochastic_quasiholder_3}
        \int_\Omega \int_{\R^d \setminus B_\delta(0)} &\abs{(\LL f)(x-y) - (\LL f)(x)}q_{\epsilon}(y) \, \mathrm{d}y \,\mathrm{d}x \\
        &\le 2 \norm{\LL f}_{L^\infty} \intf_\Omega \intf_{\R^d \setminus B_\delta(0)} q_{\epsilon}(y) dy dx \\
        &\le 2 \frac{\max\{1, \eta_0^\beta\}}{\nu_d\eta_0^d} \norm{\LL f}_\beta \intf_{\R^d \setminus B_\delta(0)} q_{\epsilon}(y) dy,
      \end{split}\end{equation}
      where the final inequality is obtained from Proposition \ref{prop:injection_l_infty}. Combining \eqref{eq:kl1_stochastic_quasiholder_1}, \eqref{eq:kl1_stochastic_quasiholder_2} and \eqref{eq:kl1_stochastic_quasiholder_3} we obtain
      \begin{equation*}
        \tnorm{\LL_\epsilon - \LL_0} \le \delta^\beta \norm{\LL}_\beta + 2\frac{\max\{1, \eta_0^\beta\}}{\nu_d\eta_0^d} \norm{\LL}_\beta \intf_{\R^d \setminus B_\delta(0)} q_{\epsilon}(y) dy.
      \end{equation*}
      As $\{q_\epsilon \}_{\epsilon > 0}$ is an approximation to the identity, taking $\epsilon \to 0$ yields
      \begin{equation*}
        \limsup_{\epsilon \to 0} \tnorm{\LL_\epsilon - \LL_0} \le \delta^\beta \norm{\LL}_\beta.
      \end{equation*}
      We conclude the proof upon recalling that $\delta$ may be chosen to be arbitrarily small.
    \end{proof}
  \end{proposition}

  \begin{proposition}[\ref{en:kl_l1_bound} for stochastic perturbations]\label{prop:kl2_stochastic_quasiholder}
    If $\{q_\epsilon\}_{\epsilon > 0}$ approximates the identity and $\{\LL_\epsilon\}_{\epsilon \ge 0}$ denotes the corresponding stochastically perturbed Perron-Frobenius operators, then for each $\epsilon \ge 0$ and $n \in \Z^+$ we have
    \begin{equation*}
      \lnorm{\LL_\epsilon^n} \le 1.
    \end{equation*}
  \end{proposition}

  \begin{lemma}\label{lemma:stochastic_quasiholder_boundedness}
    If $\{q_\epsilon\}_{\epsilon > 0}$ approximates the identity and $\{\LL_\epsilon\}_{\epsilon \ge 0}$ denotes the corresponding stochastically perturbed Perron-Frobenius operators, then for every $\epsilon > 0$ and $f \in \Va(\Omega)$ we have
    \begin{equation*}\begin{split}
      \abs{\LL_\epsilon f}_\beta \le &\left(1 + \frac{1}{\nu_d \eta_0^{d-\beta}} \sup_{0 < \eta \le \eta_0} \eta^{-\beta} m(B_\eta(\partial\Omega)) \right)\abs{\LL f}_\beta \\
      &+
      \frac{1}{\nu_d \eta_0^{d-\beta}} \left( \sup_{0 < \eta \le \eta_0} \eta^{-\beta} m(B_\eta(\partial\Omega)) \right)\lnorm{f}.
    \end{split}\end{equation*}
    \begin{proof}
      Fix $\epsilon > 0$ and $\eta \in (0, \eta_0]$. Since
      \begin{equation*}
        \osc{\LL_\epsilon f, B_\eta(x)} = \esup_{y_1,y_2 \in B_\eta(x)} \abs{(q_ \epsilon * \LL)(y_1) \bfone_\Omega(y_1) - (q_ \epsilon * \LL)(y_2) \bfone_\Omega(y_2) },
      \end{equation*}
      we may consider three (not necessarily distinct) cases when bounding $\osc{\LL_\epsilon f, B_\eta(x)}$. Depending on how many of the characteristic function terms contribute to the essential supremum, we either have $\osc{\LL_\epsilon f, B_\eta(x)} = 0$,
      \begin{equation}\label{eq:stochastic_quasiholder_boundedness_1}
        \osc{\LL_\epsilon, B_\eta(x)} = \esup_{y_1 \in B_\eta(x)} \abs{(q_\epsilon * \LL f)(y_1)},
      \end{equation}
      or
      \begin{equation}\label{eq:stochastic_quasiholder_boundedness_2}
        \osc{\LL_\epsilon f, B_\eta(x)} = \osc{q_\epsilon * \LL f, B_\eta(x)}.
      \end{equation}
      As the support of $f$ is a subset of $\Omega$, if $x \in \R^d \setminus B_\eta(\Omega)$ then $\osc{\LL_\epsilon f, B_\eta(x)} = 0$.
      By a similar argument, if \eqref{eq:stochastic_quasiholder_boundedness_1} holds, \eqref{eq:stochastic_quasiholder_boundedness_2} does not hold, and $\osc{\LL_\epsilon f, B_\eta(x)} \ne 0$, then $x \in B_\eta(\Omega)$.
      Hence,
      \begin{equation}\begin{split}\label{eq:stochastic_quasiholder_boundedness_4}
        \intf_{\R^d} \osc{\LL_\epsilon f, B_\eta(x)} dx &\le \intf_{\R^d} \osc{q_\epsilon * \LL f, B_\eta(x)} dx\\
        &+ \intf_{B_\eta(\Omega)} \enspace \esup_{y_1 \in B_\eta(x)} \abs{(q_\epsilon * \LL f)(y_1)} dx.
      \end{split}\end{equation}
      We now bound the quantity \eqref{eq:stochastic_quasiholder_boundedness_1}. As $\lnorm{\LL} \le 1$ and by Proposition \ref{prop:injection_l_infty} we have
      \begin{equation*}\begin{split}
         \abs{(q_\epsilon * \LL f)(y_1)} &= \abs{\intf_{\R^d} q_\epsilon(y)( \LL f)(y_1 -y)  dy} \\
         &\le \norm{\LL f}_\infty
         \le \frac{1}{\nu_d \eta_0^d}(\eta_0^\beta\abs{\LL f}_\beta + \lnorm{f}).
      \end{split}\end{equation*}
      Hence,
      \begin{equation}\label{eq:stochastic_quasiholder_boundedness_5}
        \intf_{B_\eta(\Omega)} \enspace \esup_{y_1 \in B_\eta(x)} \abs{(q_\epsilon * \LL f)(y_1)} dx \le \frac{m(B_\eta(\partial\Omega))}{\nu_d \eta_0^d}(\eta_0^\beta\abs{\LL f}_\beta + \lnorm{f}).
      \end{equation}
      Alternatively, to bound \eqref{eq:stochastic_quasiholder_boundedness_2} we note that
      \begin{equation*}\begin{split}
        \textup{\textrm{osc}}(q_\epsilon * \LL f&, B_\eta(x)) \\
        &=
        \esup_{y_1,y_2 \in B_\eta(x)} \abs{\intf_{\R^d} q_\epsilon(y)(( \LL f)(y_1 -y) - ( \LL f)(y_2 -y)) dy} \\
        &\le \intf_{\R^d} q_\epsilon(y) \osc{\LL f, B_\eta(x-y)} dy.
      \end{split}\end{equation*}
      By changing variables and applying Fubini-Tonelli we obtain
      \begin{equation}\begin{split}\label{eq:stochastic_quasiholder_boundedness_6}
        \int_{\R^d} &\osc{q_\epsilon * \LL, B_\eta(x)} \mathrm{d}x \\
        &\le \intf_{\R^d} \intf_{\R^d} q_\epsilon(y) \osc{\LL f, B_\eta(x-y)} dy dx \\
        &\le \intf_{\R^d} \left(\intf_{\R^d} q_\epsilon(x-y) dx\right) \osc{\LL f, B_\eta(y)} dy \\
         &\le \intf_{\R^d} \osc{\LL f, B_\eta(y)} dy.
      \end{split}\end{equation}
      Applying \eqref{eq:stochastic_quasiholder_boundedness_5} and \eqref{eq:stochastic_quasiholder_boundedness_6} to \eqref{eq:stochastic_quasiholder_boundedness_4} yields
      \begin{equation*}\begin{split}
        \intf_{\R^d} \osc{\LL_\epsilon f, B_\eta(x)} dx
        \le &\intf_{\R^d} \osc{\LL f, B_{\eta}(x)} dx\\ &+  \frac{m(B_\eta(\partial\Omega))}{\nu_d \eta_0^d}(\eta_0^\beta\abs{\LL f}_\beta + \lnorm{f}).
      \end{split}\end{equation*}
      Thus,
      \begin{equation*}
        \abs{\LL_\epsilon f}_\beta \le \abs{\LL f}_\beta +
        \frac{1}{\nu_d \eta_0^d} \left(\sup_{0 < \eta \le \eta_0} \eta^{-\beta} m(B_\eta(\partial\Omega))\right) (\eta_0^\beta\abs{\LL f}_\beta + \lnorm{f}),
      \end{equation*}
      which yields the required bound.
    \end{proof}
  \end{lemma}

  \begin{proposition}[\ref{en:kl_ly_bound} for stochastic perturbations]\label{proposition:kl3_stochastic_quasiholder}
    Under the hypotheses of Theorem \ref{thm:stability_quasiholder_stochastic} there exists $\alpha \in (0,1)$ and $C_3 > 0$ such that for every $\epsilon \ge 0$ and $f \in \Va(\Omega)$ we have
    \begin{equation*}
      \norm{\LL_\epsilon^n f}_\beta \le \alpha^n \norm{f}_\beta + C_3 \lnorm{f}.
    \end{equation*}
    \begin{proof}
      By Lemma \ref{lemma:stochastic_quasiholder_boundedness}, and Propositions \ref{prop:ly_unperturbed_quasiholder} and \ref{prop:kl2_stochastic_quasiholder}, for each $\epsilon > 0$ and $f \in \Va(\Omega)$ we have
      \begin{equation*}
        \abs{\LL_\epsilon f}_\beta \le \alpha \abs{f}_\beta + C\lnorm{f}
      \end{equation*}
      where
      \begin{equation*}
        \alpha = \gamma \left(1 + \frac{\eta_0^\beta}{\nu_d \eta_0^d} \sup_{0 < \eta \le \eta_0} \eta^{-\beta} m(B_\eta(\partial\Omega)) \right)
      \end{equation*}
      and
      \begin{equation*}
        C = D + \frac{(D+1)\eta_0^\beta}{\nu_d \eta_0^d} \left( \sup_{0 < \eta \le \eta_0} \eta^{-\beta} m(B_\eta(\partial\Omega)) \right).
      \end{equation*}
      Note that $\alpha < 1$ by the hypotheses of Theorem \ref{thm:stability_quasiholder_stochastic}. The remainder of the argument is identical to that of Proposition \ref{prop:kl_ly_quasiholder_ulam}.
    \end{proof}
  \end{proposition}

  \begin{proof}[Proof of Theorem \ref{thm:stability_quasiholder_stochastic}]
    By Lemma \ref{lemma:stochastic_quasiholder_boundedness} each $\LL_\epsilon$ is in $L(\Va(\Omega))$.
    By Propositions \ref{prop:kl1_stochastic_quasiholder}, \ref{prop:kl2_stochastic_quasiholder} and \ref{proposition:kl3_stochastic_quasiholder} the family of operators $\{ \LL_\epsilon \}_{\epsilon \ge 0}$ satisfies (KL). We also note that each $\LL_\epsilon$ is a positive operator, as each $q_\epsilon$ is non-negative. The required result then follows from Theorem \ref{thm:stability_statistical_props_quasiholder}.
  \end{proof}

\section*{Acknowledgements}
HC is supported by an Australian Government Research Training Program Scholarship and the UNSW School of Mathematics and Statistics. GF is partially supported by an Australian Research Council Discovery Project.
Both authors thank Davor Dragi\v{c}evi\'{c} for helpful conversations during the writing of this work, and to an anonymous referee for their suggested strengthening of continuity to H{\"o}lder continuity in Theorems \ref{theorem:convergence_of_eigendata_derivs}, \ref{thm:stability_of_variance} and \ref{thm:convergence_of_rate_functions}.

\appendix

\section{The proofs of Lemmas \ref{cor:conditional_exp_quasi_holder_uniform_bound} and \ref{lemma:bounded_cond_exp_quasiholder} }\label{sec:ulam_quasiholder_boundedness}

  Before discussing our strategy for proving Lemmas \ref{cor:conditional_exp_quasi_holder_uniform_bound} and \ref{lemma:bounded_cond_exp_quasiholder} we must discuss the relationship between the space $\Va(\Omega)$ and the seminorm $\abs{\cdot}_\beta$. It is noted in \cite{keller1985generalized} that while $\Va(\Omega)$ is independent of $\eta_0$, the seminorm $\abs{\cdot}_\beta$ is not. However, changing $\eta_0$ preserves the topology induced by the relevant seminorm, which will be critical to proofs in this appendix. The following lemma gives the relevant bounds.

  \begin{lemma}\label{prop:quasihold_equiv_seminorm}
    For $\zeta > 0$ and $f \in L^1(\R^d)$ let
    \begin{equation*}
      \abs{f}_{\beta,\zeta} = \sup_{0 < \eta \le \zeta} \eta^{-\beta} \intf_{\R^d} \osc{f, B_\eta(x)} dx.
    \end{equation*}
    If $0 < t \le s$ then
    \begin{equation*}
      \abs{\cdot}_{\beta,t} \le \abs{\cdot}_{\beta,s} \le S(t,s) \abs{\cdot}_{\beta,t},
    \end{equation*}
    where $S(t,s)$ denotes the minimal number of balls of radius $t$ required to cover (up to a set of measure 0) a ball of radius $s$.
    \begin{proof}
      The inequality $\abs{\cdot}_{\beta,t} \le \abs{\cdot}_{\beta,s}$ is trivial. Let $f \in L^1(\R^d)$. If
      \begin{equation}\label{eq:quasihold_equiv_seminorm_1}
        \sup_{0 < \eta \le t} \eta^{-\beta} \intf_{\R^d} \osc{f, B_\eta(x)} dx = \sup_{0 < \eta \le s} \eta^{-\beta} \intf_{\R^d} \osc{f, B_\eta(x)} dx,
      \end{equation}
      then, as $S(t,s) \ge 1$, we clearly have $\abs{f}_{\beta,s} \le  S(t,s) \abs{f}_{\beta,t}$. Alternatively, if \eqref{eq:quasihold_equiv_seminorm_1} does not hold then
      \begin{equation*}
        \sup_{0 < \eta \le t} \eta^{-\beta} \intf_{\R^d} \osc{f, B_\eta(x)} dx < \sup_{t < \eta \le s} \eta^{-\beta} \intf_{\R^d} \osc{f, B_\eta(x)} dx.
      \end{equation*}
      By the definition of $S(t,s)$ there exists $\{c_i\}_{i=1}^{S(t,s)} \subseteq \R^d$ and a set $N$ of measure 0 such that
      \begin{equation*}
        B_s(x) \setminus (N + x) \subseteq \bigcup_{i=1}^{S(t,s)} B_t(x+c_i)
      \end{equation*}
      for every $x \in \R^d$. Hence, for any $\eta \in (t, s]$ and $x \in \R^d$ we have
      \begin{equation*}
        \osc{f, B_\eta(x)} \le \osc{f, B_s(x)} \le \sum_{i=1}^{S(t,s)} \osc{f,B_t(x+c_i)}.
      \end{equation*}
      After integrating, taking the supremum and applying the definition of $\abs{\cdot}^t_\beta$ we obtain
      \begin{equation*}\begin{split}
        \sup_{t < \eta \le s} \eta^{-\beta} \intf_{\R^d} \osc{f, B_\eta(x)} dx &\le S(t,s) t^{-\beta} \intf_{\R^d} \osc{f, B_t(x)} dx \\
        &\le S(t,s) \abs{f}_{\beta,t},
      \end{split}\end{equation*}
      completing the proof.
    \end{proof}
  \end{lemma}

  We obtain Lemmas \ref{cor:conditional_exp_quasi_holder_uniform_bound} and \ref{lemma:bounded_cond_exp_quasiholder} as corollaries to the following result.

  \begin{proposition}\label{thm:conditional_exp_quasi_holder_bound}
    If $Q \in \mathcal{P}(\kappa)$ satisfies $\diam(Q) < \eta_0$, then
    \begin{equation}\label{thm:conditional_exp_quasi_holder_bound_0}
      \abs{\E_Q}_\beta \le S(\eta_0 - \diam(Q), \eta_0)\left(1 + \frac{\kappa}{\sqrt[d]{\frac{3}{2}} - 1} \right)^\beta.
    \end{equation}
  \end{proposition}

  We prove Proposition \ref{thm:conditional_exp_quasi_holder_bound} by using Lemma \ref{prop:quasihold_equiv_seminorm} to extend a bound for $\sup_{\abs{f}_\beta = 1} \abs{\E_Q f}_{\beta,\eta_0 - \diam(Q)}$ to a bound for $\abs{\E_Q}_\beta$.
  We do this by combining two bounds for
  \begin{equation*}
    \eta^{-\beta} \intf \osc{\E_Q f, B_\eta(x)} dm
  \end{equation*}
  for $\eta \in (0, \eta_0 - \diam(Q)]$:
  \begin{enumerate}
   \item We obtain the `big' $\eta$ bound by scaling the $\eta$ balls in $\funcosc$ up to $\eta+\diam(Q)$ balls. This bound is useful for large $\eta$, but grows unboundedly as $\eta$ vanishes. We obtain this bound in Lemma \ref{prop:big_eps_bound}.
   \item We obtain the `small' $\eta$ bound by using the geometry of the elements of $Q$ to quantify the decay of the measure of the support of $\osc{\E_Q f, B_\eta(\cdot)}$ as $\eta$ vanishes. This bound is used for $\eta$ arbitrarily close to 0. Obtaining this bound is more complicated and is developed in Lemmas \ref{lemma:osc_support_boundary} - \ref{prop:small_eps_osc_bound}.
  \end{enumerate}
  Before proving either of these bounds we derive an explicit expression for $\intf \osc{\E_Q f, B_\eta(x)} dm$.

\begin{lemma}\label{lemma:conditional_exp_osc_sum}
  Let $Q \in \mathcal{P}(\kappa)$, define $Q' = Q \cup \{ \Omega^c\}$, and for each $\eta > 0$ and $x \in \R^d$ let
  \begin{equation*}
    N(x,\eta) = \{ J \in Q' : B_\eta(x) \cap J \ne \emptyset \}.
  \end{equation*}
  For each $\eta > 0$, $f \in \Va(\Omega)$ and $S \subseteq Q' $ let
  \begin{equation*}
    M_S(f) = \max_{J,K \in S} \abs{\avg{f}{J} - \avg{f}{K}} \quad \text{and}  \quad A_{S,\eta} = \{ x \in \R^d : N(x, \eta) = S \}.
  \end{equation*}
  Then each $A_{S,\eta}$ is measurable, and for every $x \in \R^d$
  \begin{equation}\label{eq:conditional_exp_osc_sum_1}
    \osc{\E_Q f,B_\eta(x)} = M_{N(x,\eta)}(f).
  \end{equation}
  Hence,
  \begin{equation}\label{eq:conditional_exp_osc_sum_2}
    \intf_{\R^d} \osc{\E_Q f, B_\eta(x)} dx = \sum_{S \subseteq Q' } m(A_{S,\eta}) M_{S}(f).
  \end{equation}
  \begin{proof}
    For every $J \in Q$ the equality
    \begin{equation*}
      \{x \in \R^d : B_\eta(x) \cap J \ne \emptyset \} = \bigcup_{y \in J} B_\eta(y),
    \end{equation*}
    implies that both sets are open, and therefore measurable. Recalling from Definition \ref{def:regular_admissible_partitions} that $Q$ is finite and noting the equality
    \begin{equation*}\begin{split}
      A_{S,\eta} = &\left( \bigcap_{J \in S} \{x \in \R^d : B_\eta(x)
      \cap J \ne \emptyset \}  \right) \\
      &\bigcap \left( \bigcap_{K \in Q' \setminus S} \{x \in \R^d : B_\eta(x) \cap K = \emptyset \}  \right),
    \end{split}\end{equation*}
    we conclude that each $A_{S,\eta}$ is measurable. Considering the definition of $N(x,\eta)$, we note that the family of sets $\{ A_{S, \eta} : S \subseteq Q'\}$ partitions $\R^d$. Hence,
    \begin{equation*}
      \intf_{\R^d} \osc{\E_Q f, B_\eta(x)} dx = \sum_{S \subseteq Q' } \intf_{A_{S,\eta}} \osc{\E_Q f, B_\eta(x)} dx,
    \end{equation*}
    where we note that the sum on the right-hand side is well defined as only finitely many terms are ever non-zero. Thus, in order to prove \eqref{eq:conditional_exp_osc_sum_2} it suffices to prove \eqref{eq:conditional_exp_osc_sum_1}. If $N(x,\eta) = S$ then
    \begin{equation*}
      (\E_Q f)(B_\eta(x)) = \left\{ \avg{f}{J} : J \cap B_\eta(x) \ne \emptyset \right\} = \left\{ \avg{f}{J} : J \in S \right\},
    \end{equation*}
    which is finite as $Q'$ is finite. By applying the definition of $\funcosc$, we find that
    \begin{equation*}
      \osc{\E_Q f, B_\eta(x)} = \max_{J,K \in S} \abs{\avg{f}{J} - \avg{f}{K}} = M_S(f),
    \end{equation*}
    which is exactly \eqref{eq:conditional_exp_osc_sum_2}.
  \end{proof}
\end{lemma}

We may now obtain the `big' $\eta$ bound.

\begin{lemma}\label{prop:big_eps_bound}
  If $Q \in \mathcal{P}(\kappa)$ then for each $f \in \Va(\Omega)$ and $\eta > 0$ we have
  \begin{equation}\label{eq:big_eps_bound_1}
    \intf_{\R^d} \osc{\E_Q f, B_\eta(x)} dx \le \intf_{\R^d} \osc{f, B_{\eta+\diam(Q)}(x)} dx.
  \end{equation}
  Furthermore, if $\diam(Q) < \eta_0$ and $\eta \in (0, \eta_0 - \diam(Q)]$ then
  \begin{equation}\label{eq:big_eps_bound_2}
    \eta^{-\beta} \intf_{\R^d} \osc{\E_Q f, B_\eta(x)} dx \le \left(1 +\frac{\diam(Q)}{\eta}\right)^\beta \abs{f}_\beta.
  \end{equation}
  \begin{proof}
    Fix $x \in \R^d$. By Lemma \ref{lemma:conditional_exp_osc_sum}, we have
    \begin{equation*}
      \osc{\E_Q f, B_\eta(x)} = \max_{J,K \in N(x, \eta)} \abs{\avg{f}{J} -\avg{f}{K}}.
    \end{equation*}
    Suppose that $y \in I$ for some $I \in N(x, \eta) \setminus \{\Omega^c\}$. By the definition of $N(x, \eta)$ there exists $z \in I$ such that $\abs{z - x} < \eta$. Since $\abs{z-y} \le \diam(Q)$ we have $\abs{y-x} < \eta + \diam(Q)$. Hence
    \begin{equation}\label{eq:big_eps_bound_3}
      \bigcup_{I \in N(x,\eta) \setminus \{\Omega^c\}} I \subseteq B_{\eta + \diam(Q)}(x).
    \end{equation}
    Now suppose that $J,K \in N(x,\eta)$. In the case where $J,K \in N(x,\eta)\setminus \{\Omega^c\}$ the inclusion \eqref{eq:big_eps_bound_3} implies that for almost every $(y_1, y_2) \in J \times K$ we have
    \begin{equation*}
      \abs{f(y_1) - f(y_2)} \le \osc{f, B_{\eta + \diam(Q)}(x)}.
    \end{equation*}
    By taking expectations with respect to $y_1$ over $J$ and $y_2$ over $K$ we obtain
    \begin{equation}\label{eq:big_eps_bound_4}
      \abs{\avg{f}{J} - \avg{f}{K}} \le \osc{f, B_{\eta + \diam(Q)}(x)}.
    \end{equation}
    Alternatively, if either $J$ or $K$ is equal to $\Omega^c$ then
    \begin{equation}\begin{split}\label{eq:big_eps_bound_5}
      \abs{\avg{f}{J} - \avg{f}{K}} = \max \left\{\abs{\avg{f}{J}}, \abs{\avg{f}{K}}\right\} &\le \max_{I \in N(x, \eta) \setminus \Omega^c} \abs{\avg{f}{I}} \\
      &\le \esup_{y \in B_{\eta + \diam(Q)}} \abs{f(y)},
    \end{split}\end{equation}
    where we obtain the last inequality by using \eqref{eq:big_eps_bound_3}. Noting that the set $B_{\eta + \diam(Q)}(x) \cap \Omega^c$ has non-zero measure, we have
    \begin{equation}\label{eq:big_eps_bound_6}
      \esup_{y \in B_{\eta + \diam(Q)}(x)} \abs{f(y)} \le \osc{f,B_{\eta + \diam(Q)}(x)}.
    \end{equation}
    By combining \eqref{eq:big_eps_bound_5} and \eqref{eq:big_eps_bound_6} we obtain \eqref{eq:big_eps_bound_4} for the case where either $J$ or $K$ is equal to $\Omega^c$. As $J$ and $K$ were arbitrary elements of $N(x,\eta)$ this implies that
    \begin{equation*}
      \osc{\E_Q  f, B_\eta(x)} = \max_{I,J \in N(x,\eta)} \abs{\avg{f}{I} -\avg{f}{J}} \le \osc{f, B_{\eta + \diam(Q)}(x)}.
    \end{equation*}
    By integrating with respect to $x$ over $\R^d$ we obtain \eqref{eq:big_eps_bound_1}. To prove \eqref{eq:big_eps_bound_2}, suppose that $\diam(Q) < \eta_0$. If $\eta \in (0, \eta_0 -\diam(Q)]$ then $\eta + \diam(Q) \in (0, \eta_0]$ and so the definition of $\abs{\cdot}_\beta$ implies that
    \begin{equation*}
      \intf_{\R^d} \osc{f, B_{\eta+\diam(Q)}(x)} dx \le (\eta + \diam(Q))^{\beta} \abs{f}_\beta.
    \end{equation*}
    Thus
    \begin{equation*}\begin{split}
      \eta^{-\beta} \intf_{\R^d} \osc{\E_Q f, B_{\eta}(x)} dx &\le \eta^{-\beta} \intf_{\R^d} \osc{f, B_{\eta+\diam(Q)}(x)} dx \\
      &\le \left(1 +\frac{\diam(Q)}{\eta}\right)^\beta \abs{f}_\beta.
    \end{split}\end{equation*}
  \end{proof}
\end{lemma}

We will now pursue the `small' $\eta$ bound.

\begin{lemma} \label{lemma:osc_support_boundary}
  Let $Q \in \mathcal{P}(\kappa)$ and let $S \subseteq Q'$ satisfy $\abs{S} > 1$. If $I \in S \setminus \{ \Omega^c\}$ and $\eta > 0$ then $A_{S, \eta} \subseteq B_\eta(\partial I)$.
  \begin{proof}
    The claim is trivially true if $A_{S,\eta}$ is empty,  henceforth we assume that it is not. Let $x \in A_{S, \eta}$. We distinguish between two cases: either $x \in I$ or $x \notin I$. Suppose that $x \in I$. As $\abs{S} > 1$ and $N(x,\eta) = S$ there exists some $J \in Q' \setminus \{I\}$ such that $B_\eta(x) \cap J \ne \emptyset$.
    Actually, as the closure of the interior of $J$ is $J$, we have $B_\eta(x) \cap \intr(J) \ne \emptyset$. In this case let $y \in B_\eta(x) \cap \intr(J)$; as $J$ and $I$ are convex elements of a measurable partition we have $J \cap I \subseteq \partial J \cap \partial I$ and so $y \notin I$. Alternatively, if $x \notin I$, then let $y \in B_\eta(x) \cap I$, which is non-empty by a similar argument.
    In both cases we have a pair of points in $A_{S, \eta}$: one in $I$ and the other not. Recalling that elements of $Q'$ have non-empty interior and then considering the line segment that joins $x$ and $y$, it is straightforward to verify that there exists some $z \in \partial I$ on this line segment. Clearly $\abs{x-z} < \eta$ and so $x \in B_\eta(\partial I)$, which completes the proof.
  \end{proof}
\end{lemma}

\begin{lemma}\label{prop:small_eps_size_bound}
  Let $Q \in \mathcal{P}(\kappa)$. If $\eta > 0$ and $S \subseteq Q'$ is such that $\abs{S} > 1$ and $m(A_{S,\eta}) > 0$, then for each $f \in \Va(\Omega)$ we have
  \begin{equation*}\begin{split}
    M_S(f) \le \max_{\substack{J,K \in S \\ J\ne K}}\Bigg(&\intf_J \frac{\osc{f, B_{\eta + \diam(Q)}(x)}}{m(J)} dx \\
    &+ \intf_K \frac{ \osc{f, B_{\eta + \diam(Q)}(x)}}{m(K)} dx \Bigg).
  \end{split}\end{equation*}
  \begin{proof}
    Let $J,K \in S$ be partition elements satisfying
    \begin{equation*}
      M_S(f) = \abs{\avg{f}{J} - \avg{f}{K}}.
    \end{equation*}
    We may assume that $J \ne K$, as this case does not contribute to the maximum. Let us first consider the case where $\Omega^c \in \{J,K\}$; without loss of generality let $K = \Omega^c$. For every $j \in J$ we have $J \subseteq B_{\eta + \diam(Q)}(j)$.
    Hence, as $B_{\eta + \diam(Q)}(j) \cap \Omega^c$ has non-empty interior, and therefore non-zero measure, for almost every $j, j' \in J$ and $k' \in B_{\eta + \diam(Q)}(j) \cap \Omega^c$ we have
    \begin{equation*}
      \abs{f(j')} = \abs{f(j') - f(k')} \le \osc{f,B_{\eta + \diam(Q)}(j)}.
    \end{equation*}
    Taking expectations with respect to $j'$ and $j$ over $J$ yields
    \begin{equation*}
      M_S(f) = \abs{\avg{f}{J}} \le \intf_J \frac{\osc{f,B_{\eta + \diam(Q)}(x)}}{m(J)} dx,
    \end{equation*}
    which implies the required conclusion.
    Alternatively suppose that neither $J$ nor $K$ is equal to $\Omega^c$. Fix $j \in J$ and $k \in K$. For any $j' \in J$ we have $\abs{j - j'} \le \diam(Q)$ and so $j' \in B_{\eta + \diam(Q)}(j)$. Similarly, for every $k' \in K$ we have $k' \in B_{\eta + \diam(Q)}(k)$. As $m(A_{S,\eta}) > 0$, we know that $A_{S,\eta} \ne \emptyset$. For $z \in A_{S,\eta}$ the intersection $B_\eta(z) \cap J$ is non-empty and so $z \in B_{\eta + \diam(Q)}(j)$. Similarly, $z \in B_{\eta + \diam(Q)}(k)$. Hence, for almost every $j' \in J$ and $k' \in K$,
    \begin{equation*}\begin{split}
      \abs{f(j') - f(k')} &\le \abs{f(j') - f(z)} + \abs{f(k') - f(z)}\\
      &\le \osc{f, B_{\eta + \diam(Q)}(j)} + \osc{f, B_{\eta + \diam(Q)}(k)}.
    \end{split}\end{equation*}
    By taking the expectation with respect to $j'$ over $J$ and $k'$ over $K$, we find
    \begin{equation}\label{eq:small_eps_size_bound_1}
      \abs{\avg{f}{J} - \avg{f}{K}} \le \osc{f, B_{\eta + \diam(Q)}(j)} + \osc{f, B_{\eta + \diam(Q)}(k)}.
    \end{equation}
    Since \eqref{eq:small_eps_size_bound_1} holds for every $j \in J$ and $k \in K$, we may take expectations again to obtain
    \begin{equation*}
      \abs{\avg{f}{J} - \avg{f}{K}} \le  \intf_J \frac{\osc{f, B_{\eta + \diam(Q)}(x)}}{m(J)}  dx + \intf_K \frac{\osc{f, B_{\eta + \diam(Q)}(x)}}{m(K)}  dx.
    \end{equation*}
    We obtain the required inequality by taking the maximum over all distinct pairs of $J,K \in S$.
  \end{proof}
\end{lemma}

Combining the previous two results yields the `small' $\eta$ bound of Lemma \ref{prop:small_eps_osc_bound}.

\begin{lemma}\label{prop:small_eps_osc_bound}
  Let $Q \in \mathcal{P}(\kappa)$. If $\diam(Q) < \eta_0$, $\eta \in (0, \eta_0 - \diam(Q)]$ and $f \in \Va(\Omega)$ then
  \begin{equation*}
    \eta^{-\beta} \intf_{\R^d} \osc{\E_Q f, B_\eta(x)} dx \le \left(\max_{I \in Q} \frac{m(B_\eta(\partial I))}{m(I)} \right) \left(1 + \frac{\diam(Q)}{\eta}\right)^\beta \abs{f}_\beta.
  \end{equation*}
  \begin{proof}
    By Lemma \ref{lemma:conditional_exp_osc_sum} we have
    \begin{equation}\label{eq:small_eps_osc_bound_1}
      \intf_{\R^d} \osc{\E_Q f, B_\eta(x)} dx = \sum_{S \subseteq Q' } m(A_{S,\eta}) M_{S}(f).
    \end{equation}
    Let $G = \{ S \subseteq Q' : \abs{S} > 1, m(A_{S,\eta}) > 0 \}$. Since $m(A_{S,\eta}) M_{S}(f) = 0$ if $S \notin G$ we may restrict the sum in \eqref{eq:small_eps_osc_bound_1}:
    \begin{equation}\label{eq:small_eps_osc_bound_2}
      \intf_{\R^d} \osc{\E_Q f, B_\eta(x)} dx = \sum_{S \in G} m(A_{S,\eta}) M_{S}(f).
    \end{equation}
    Applying Lemma \ref{prop:small_eps_size_bound} to each of the terms in \eqref{eq:small_eps_osc_bound_2} yields
    \begin{equation}\begin{split}\label{eq:small_eps_osc_bound_3}
      &\intf_{\R^d} \osc{\E_Q f, B_\eta(x)} dx \\
      &\begin{split}
        \le \sum_{S \in G} m(A_{S,\eta}) \max_{J,K \in S, J \ne K} \biggl( &\intf_J \frac{\osc{f, B_{\eta + \diam(Q)}(x)}}{m(J)}  dx  \\
        &+  \intf_K \frac{\osc{f, B_{\eta + \diam(Q)}(x)}}{m(K)}  dx \biggr).\end{split}
    \end{split}\end{equation}
    By rearranging the terms in \eqref{eq:small_eps_osc_bound_3} to sum over elements of $Q$ we obtain
    \begin{equation}\begin{split}\label{eq:small_eps_osc_bound_4}
      &\int_{\R^d}\osc{\E_Q f, B_\eta(x)} \mathrm{d}x \\
      &\le \sum_{I \in Q} \frac{\sum_{S \in G, I \in S} m(A_{S,\eta})}{m(I)} \intf_I \osc{f, B_{\eta + \diam(Q)}(x)} dx \\
      &\le \left(\max_{I \in Q} \frac{\sum_{ S\in G, I \in S} m(A_{S,\eta})}{m(I)} \right) \intf_{\R^d} \osc{f, B_{\eta + \diam(Q)}(x)} dx,
    \end{split}\end{equation}
    where we omit the case of $I = \Omega^c$, as it does not contribute to the sum. Since the sets $\{A_{S,\eta}\}_{S \subseteq Q'}$ are disjoint, Lemma \ref{lemma:osc_support_boundary} implies that
    \begin{equation*}
      \sum_{ S \in G, I \in S } m(A_{S,\eta}) \le m(B_\eta(\partial I)).
    \end{equation*}
    Thus,
    \begin{equation*}\begin{split}
      \int_{\R^d} &\osc{\E_Q f, B_\eta(x)} \mathrm{d}x \\
      &\le \left(\max_{I \in Q} \frac{m(B_\eta(\partial I))}{m(I)} \right) \intf_{\R^d} \osc{f, B_{\eta + \diam(Q)}(x)} dx.
    \end{split}\end{equation*}
     The required inequality follows by applying the definition of $\abs{\cdot}_\beta$.
  \end{proof}
\end{lemma}

Before proving Proposition \ref{thm:conditional_exp_quasi_holder_bound} we require a technical lemma for an inequality from convex geometry. For $U,V \subseteq \R^d$ the Minkowski sum of $U$ and $V$ is denoted by $U + V$ and equal to $\{u + v : u \in U, v \in V\}$; for basic properties we refer to \cite[Section 6.1]{gruber2010convex}.

\begin{lemma}\label{lemma:convex_inside_outside}
  If $I$ is a compact convex polytope then for every $\eta > 0$ we have
  \begin{equation*}
    m(B_\eta(\partial I) \cap I) \le m(B_\eta(\partial I)  \cap I^c).
  \end{equation*}
  \begin{proof}
    Let $m_{d-1}$ denote $d-1$ dimensional Lebesgue measure. By Steiner's formula \cite[Theorem 6.6]{gruber2010convex} there exists a polynomial $p_I$ with positive coefficients and of degree $d$ such that $m(B_\eta(I)) = p_I(\eta)$. The constant coefficient of $p_I$ is clearly $m(I)$, while the coefficient of the linear term is $m_{d-1}(\partial I)$ i.e. the surface area of $I$. Note that $m(B_\eta(\partial I)  \cap I^c) = p_I(\eta) - m(I)$. We will prove that $m(B_\eta(\partial I) \cap I) \le \eta m_{d-1}(\partial I)$. Since $p_I$ has degree greater than or equal to $2$ and positive coefficients, it follows that
    \begin{equation*}
      m(B_\eta(\partial I) \cap I) \le \eta m_{d-1}(\partial I) \le p_I(\eta) - m(I) \le m(B_\eta(\partial I)  \cap I^c),
    \end{equation*}
    and would therefore complete the proof.

    Let $\mathcal{F}(I)$ denote the set of set of facets of $I$. Clearly $m_{d-1}(\partial I) = \sum_{F \in \mathcal{F}(I)} m_{d-1}(F)$. Let $y \in B_\eta(\partial I) \cap I$ and denote by $F$ the (possibly not unique) facet in $\mathcal{F}(I)$ that minimises the distance from $y$ to $\partial I$. Let $x$ be the point on $F$ attaining said minimum.
    If $x-y$ is not normal to $F$ then the ball $B_{\abs{x-y}}(y)$ is not tangent to $F$ and so there exists $z \in B_{\abs{x-y}}(y) \cap I^c$. The line segment from $y$ to $z$ must intersect $\partial I$ at some point that is strictly closer to $y$ than $x$, which contradicts $x$ minimising the distance from $y$ to $\partial I$. Hence, $x-y$ must be normal to $F$ and so $y \in F + [0,\eta] \times n_F$, where $n_F$ is the inward facing unit normal vector to $F$. This implies that
    \begin{equation*}
      B_\eta(\partial I) \cap I \subseteq \bigcap_{F \in \mathcal{F}(I)} F + [0,\eta] \times n_F
    \end{equation*}
    and so $m(B_\eta(\partial I) \cap I) \le \eta \sum_{F \in \mathcal{F}(I)} m_{d-1}(F) = \eta m_{d-1}(\partial I)$ as required.
  \end{proof}
\end{lemma}

\begin{proof}[The proof of Proposition \ref{thm:conditional_exp_quasi_holder_bound}]
  We begin by bounding
  \begin{equation*}
    \sup_{\abs{f}_\beta = 1} \abs{\E_Q f}_{\beta,\eta_0 - \diam(Q)}.
  \end{equation*}
  Let $b: \R \to \R$ be defined by
  \begin{equation*}
    b(\eta) = \left(1 + \frac{\diam(Q)}{\eta}\right)^\beta.
  \end{equation*}
  By taking the minimum of the bounds in Lemmas \ref{prop:big_eps_bound} and \ref{prop:small_eps_osc_bound} we have
  \begin{equation}\label{eq:conditional_exp_quasi_holder_bound_1}
    \abs{\E_Q f}_{\beta,\eta_0 - \diam(Q)}
    \le \sup_{0 < \eta \le \eta_0} \min\left\{\max_{I \in Q} \frac{m(B_\eta(\partial I))}{m(I)}, 1\right\} b(\eta)  \abs{f}_\beta.
  \end{equation}
  We will now bound $\max_{I \in Q} \frac{m(B_\eta(\partial I))}{m(I)}$. Lemma \ref{lemma:convex_inside_outside} implies that for any $I \in Q$ we have
  \begin{equation*}
    \frac{m(B_{\eta}(\partial I))}{m(I)}\le \frac{2m(B_{\eta}(\partial I) \cap I^c)}{m(I)}.
  \end{equation*}
  Noting that $B_{\eta}(\partial I) \cap I^c = B_{\eta}(I) \setminus I$ and $B_\eta(I) = I + (\eta/2)B_{1}(0)$, we obtain
  \begin{equation}\label{eq:conditional_exp_quasi_holder_bound_2}
    \frac{m(B_{\eta}(\partial I))}{m(I)} \le 2 \frac{m(I + (\eta/2)B_{1}(0)) -m(I)}{m(I)}.
  \end{equation}
  Let $B_I$ be a ball inscribed in $I$ of maximal volume. Then, by scaling and possibly translating by some vector $v_I \in \R^d$, we find that $B_{1}(0) \subseteq \frac{2}{\diam(B_I)} I + v_I$. Consequently
  \begin{equation}\label{eq:conditional_exp_quasi_holder_bound_3}
    m\left(I + \frac{\eta}{2}B_{1}(0)\right) \le m\left(I + \frac{\eta}{\diam(B_I)}I\right) = \left( 1 + \frac{\eta}{\diam(B_I)} \right)^d m(I).
  \end{equation}
  Applying \eqref{eq:conditional_exp_quasi_holder_bound_3} to \eqref{eq:conditional_exp_quasi_holder_bound_2}, and recalling that $1/ \diam(B_I) \le \kappa/\diam(Q)$ as $Q \in \mathcal{P}(\kappa)$, we find that
  \begin{equation}\label{eq:conditional_exp_quasi_holder_bound_4}
    \frac{m(B_{\eta}(\partial I))}{m(I)} \le 2 \left( 1 + \frac{\eta}{\diam(B_I)} \right)^d - 2 \le 2 \left( 1 + \frac{\kappa\eta}{\diam(Q)} \right)^d - 2.
  \end{equation}
  By applying \eqref{eq:conditional_exp_quasi_holder_bound_4} to \eqref{eq:conditional_exp_quasi_holder_bound_1} we obtain
  \begin{equation}\begin{split}\label{eq:conditional_exp_quasi_holder_bound_5}
    &\lvert \E_Q f \rvert_{\beta,\eta_0 - \diam(Q)} \\
    &\le \sup_{0 < \eta \le \eta_0} \min\left\{\left(2 \left( 1 + \frac{\kappa\eta}{\diam(Q)} \right)^d - 2\right) b(\eta), b(\eta)\right\} \abs{f}_\beta.
  \end{split}\end{equation}
  It is clear that $b$ is monotonically decreasing.
  Note that
  \begin{equation}\label{eq:conditional_exp_quasi_holder_bound_6}
    \begin{split}
    2 \Bigg(\Bigg( 1 &+ \frac{\kappa\eta}{\diam(Q)} \Bigg)^d  - 1\Bigg)b(\eta) \\
    &= 2 \eta^{-\beta}\left(\left( 1 + \frac{\kappa\eta}{\diam(Q)} \right)^d - 1\right)(\eta + \diam(Q))^\beta.
  \end{split}
  \end{equation}
  The map $\eta \mapsto (\eta + \diam(Q))^\beta$ is clearly monotonically increasing on $(0, \eta_0]$.
  As $d \ge 2$ and $\beta \in (0, 1]$, the map
  \begin{equation*}
    \eta \mapsto \eta^{-\beta} \left(\left( 1 + \frac{\kappa\eta}{\diam(Q)} \right)^d - 1\right)
  \end{equation*}
  is monotonically increasing on $(0, \eta_0]$ too. Thus the left-hand side of \eqref{eq:conditional_exp_quasi_holder_bound_6} is monotonically increasing.
  Since both $b$ and the left-hand side of \eqref{eq:conditional_exp_quasi_holder_bound_6} are continuous on $(0, \eta_0]$, $b$ is monotonically decreasing and the left-hand side of \eqref{eq:conditional_exp_quasi_holder_bound_6} is monotonically increasing, it follows that if $\eta' \in (0, \infty)$ solves
  \begin{equation}\label{eq:conditional_exp_quasi_holder_bound_7}
    2 \left( 1 + \frac{\kappa\eta'}{\diam(Q)} \right)^d - 2 = 1,
  \end{equation}
  then
  \begin{equation*}
    \sup_{0 < \eta \le \eta_0} \min\left\{\left(2 \left( 1 + \frac{\kappa\eta}{\diam(Q)} \right)^d - 2\right) b(\eta), b(\eta)\right\} \le b(\eta').
  \end{equation*}
  Solving \eqref{eq:conditional_exp_quasi_holder_bound_7} yields
  \begin{equation*}
    \frac{\diam(Q)}{\eta'} = \frac{\kappa}{\sqrt[d]{\frac{3}{2}} - 1}.
  \end{equation*}
  By substituting this into \eqref{eq:conditional_exp_quasi_holder_bound_5} we obtain the bound
  \begin{equation*}
    \abs{\E_Q f}_{\beta,\eta_0 - \diam(Q)} \le \left(1 + \frac{\kappa}{\sqrt[d]{\frac{3}{2}} - 1} \right)^\beta \abs{f}_\beta.
  \end{equation*}
  Applying Lemma \ref{prop:quasihold_equiv_seminorm} yields the required bound.
\end{proof}

With Proposition \ref{thm:conditional_exp_quasi_holder_bound} in hand we may now prove Lemmas \ref{cor:conditional_exp_quasi_holder_uniform_bound} and \ref{lemma:bounded_cond_exp_quasiholder}.

\begin{proof}[Proof of Lemma \ref{cor:conditional_exp_quasi_holder_uniform_bound}]
  As $\lim_{\epsilon \to 0} \diam(Q_\epsilon) = 0$ there exists $\epsilon_2 >0$ such that for every $\epsilon \in (0, \epsilon_2]$ we have $\diam(Q_\epsilon) < \eta_0$ and
  \begin{equation*}
    1 + \diam(Q_\epsilon)/(\eta_0 - \diam(Q_\epsilon)) < \sqrt{d/(d-1)}.
  \end{equation*}
  By \cite[Section 8.5, page 236]{boroczky2004finite}, this implies
  \begin{equation*}
    S(1, 1 + \diam(Q_\epsilon)/(\eta_0 - \diam(Q_\epsilon))) = S(\eta_0 - \diam(Q_\epsilon), \eta_0) \le 2d.
  \end{equation*}
  The desired conclusion follows by Proposition \ref{thm:conditional_exp_quasi_holder_bound}.
\end{proof}

\begin{proof}[Proof of Lemma \ref{lemma:bounded_cond_exp_quasiholder}]
  If $\diam(Q) < \eta_0$ then $\abs{\E_Q}_\beta < \infty$ by Proposition \ref{thm:conditional_exp_quasi_holder_bound}. Alternatively, if $\diam(Q) \ge \eta_0$, then repeatedly applying Lemma \ref{prop:quasihold_equiv_seminorm} yields
  \begin{equation*}\begin{split}
    \abs{\E_Q}_\beta &= \sup_{\abs{f}_\beta \le 1 } \abs{\E_Q f}_\beta \\
    &\le \sup \{ \abs{\E_Q f}_{\beta,2 \diam(Q)} : \abs{f}_{\beta,2\diam(Q)} \le S(\eta_0, 2\diam(Q)) \}\\
    &\le S(\eta_0, 2\diam(Q)) \abs{\E_Q}_{\beta,2 \diam(Q)},
  \end{split}\end{equation*}
  which is finite by Proposition \ref{thm:conditional_exp_quasi_holder_bound} applied to the seminorm $\abs{\cdot}_{\beta,2 \diam(Q)}$ (i.e. when $\eta_0 = 2 \diam(Q)$).
  In either case we have $\abs{\E_Q}_\beta < \infty$ and so, as $\lnorm{\E_Q } =1$, we have $\norm{\E_Q}_\beta < \infty$ too. As $Q$ partitions $\Omega$, for every $f \in \Va(\Omega)$ the support of $\E_Q f$ is a subset of $\Omega$. Hence $\E_Q f \in \Va(\Omega)$ for every $f \in \Va(\Omega)$ and so $\E_Q  \in L(\Va(\Omega))$.
\end{proof}

\section{Code}
\label{app:c}

\subsection{Variance}

\begin{lstlisting}[caption = {This function centres the observable \textit{obs} (defined by an anonymous MATLAB function e.g.\ the code snippet in Section \ref{sec:estimating_variance_1d}) and computes the required first and second derivatives at zero to estimate the variance.}]
function [ddLam,v,dv,dlam,ddlam]=variance(P,obs),

%P is a row-stochastic matrix
%obs is a pre-defined anonymous function representing the observable
%lam is the leading eigenvalue
%v is the leading eigenfunction
%dv is dv/dtheta, where theta is the twist parameter
%dlam is dlam/dtheta
%ddlam is d^2lam/dtheta^2
%ddLam is d^2Lam/dtheta^2

%% find v and normalise appropriately
n=size(P,1);
phi=ones(n,1)/n;
[v,~]=eigs(P',1);
v=v/sum(v)*n;

%% centre observable g
x=[1/(2*n):1/(n):1-1/(2*n)]';
g=obs(x);
g=g-g'*v/n;  %ensure g has mean zero by subtracting the mean

%% estimate dlam and dv using 1.*v=1 and 1.*dv=0

A=[P'-speye(n)  -v;  ones(1,n) 0];
b=[-P'*(g.*v);  0];
y=A\b;
dv=y(1:n);
dlam=y(n+1);

%% compute d^2lam/dtheta^2 and d^2Lam/dtheta^2

ddlam=((g.^2)'*v+2*g'*dv)/n;
ddLam=(ddlam-dlam^2);
\end{lstlisting}

\subsection{Rate function}

\begin{lstlisting}[caption = {This function centres the observable \textit{obs} (defined by an anonymous MATLAB function e.g.\ the code snippet in Section \ref{sec:est_rate_func}), and performs the required minimisation to evaluate the rate function at points specified in the vector $s$.}]
function [r,optz] = rate_function(s,P,obs)

%P is a row-stochastic matrix
%obs is a pre-defined anonymous function representing the observable
%s is a vector of arguments of the rate function

%% calc acim for centering observable.
[v0,~]=eigs(P',1);
v0=v0/sum(v0);

%% set up objects to pass to legendre_function.m
n=length(P);
[I,J,V]=find(P);
xpts=(I-.5)/n;
xptsorig=1/(2*n):1/n:1-1/(2*n);
gmean=obs(xptsorig)*v0;

%% set up arrays for r and optz and set optimisation options
r=zeros(length(s),1);
optz=r;
options = optimoptions('fminunc','Algorithm','quasi-newton','SpecifyObjectiveGradient',true,'OptimalityTolerance',1e-6);

%% initial seed point for minimisation
z0=0;

%% evaluate rate function at points specified in s
for i=1:length(s),
    minfun=@(z)legendre_function(z,P,v0,s(i),obs,n,I,J,V,xpts,gmean);
    [optz(i),r(i),~,~]=fminunc(minfun,z0,options);
    z0=optz(i);   %use previous optimum for next initialisation.
end

r=-r;
\end{lstlisting}

\begin{lstlisting}[caption = {This function evaluates the ``Legendre function'' (the function to be minimised) and its derivative. This requires twisting the matrix $P$ by $z$ and then computing the leading eigenvalue and eigenvector of the twisted matrix.}]
function [f,df] = legendre_function(z,P,v0,s,fun,n,I,J,V,xpts,gmean)

%evaluate 'legendre' function with fixed parameter s, maximising over z.

%% twist P by z
gvec=z*(fun(xpts)-gmean);
Vtwist=V.*exp(gvec);
Ptwist=sparse(I,J,Vtwist);

%% Calculate objective f
[v,lam]=powermethod(Ptwist,v0,1);
f=log(lam)-z*s;

%% Calculate gradient df
if nargout > 1 % gradient required
    v=v/sum(v);
    [phi,lam]=powermethod(Ptwist',ones(n,1),1);
    phi=phi/(phi'*v);
    gvecbasic=fun(1/(2*n):1/n:1-1/(2*n))-gmean;
    dlam=lam*phi'*(gvecbasic'.*v);
    df=dlam/lam-s;
end
\end{lstlisting}

\begin{lstlisting}[caption = {Estimation of the leading eigenvalue and eigenvector by repeated iteration.}]
function [v1,lam1]=powermethod(P,v0,lam0),

%P is a row stochastic matrix
%v0 is an initial (guessed) eigenvector
%lam0 is an initial (guessed) eigenvalue)

v0=v0/sum(v0);
v1=P'*v0;
lam1=sum(v1);
while abs(lam1-lam0)>1e-15,
    lam0=lam1;
    v0=v1/lam1;
    v1=P'*v0;
    lam1=sum(v1);
end
\end{lstlisting}

% Bibliography

\end{document}